\documentclass[a4paper,11pt]{amsart} 

\usepackage{amsmath} 
\usepackage{amssymb}
\usepackage{amsthm} 
\usepackage{graphicx} 
\usepackage{enumerate}
\usepackage{amsfonts} 
\usepackage{ulem} 
\usepackage{float} 
\usepackage{hyperref} 
\usepackage{cite} 
\usepackage[dvipsnames]{xcolor} 
\usepackage{tikz} 
\usepackage{tikz-cd} 
\usepackage{mathabx} 
\usepackage{faktor} 
\usepackage{array} 
\usepackage{transparent} 
\usepackage{mathdots} 
\usepackage{adjustbox} 
\usepackage{thmtools} 
\usepackage{thm-restate} 
\usepackage[margin=1.25in]{geometry} 
\usepackage{mathtools} 

\usepackage{caption}
\usepackage{subcaption}
\usepackage{geometry}

\newcommand{\N}{\mathbb{N}}
\newcommand{\Z}{\mathbb{Z}}
\newcommand{\R}{\mathbb{R}}

\usetikzlibrary{decorations.markings} 

\theoremstyle{plain} 
\newtheorem{thm}{Theorem}[subsection] 
\newtheorem{lemma}[thm]{Lemma} 
\newtheorem{prop}[thm]{Proposition} 
\newtheorem{cor}[thm]{Corollary} 

\theoremstyle{definition} 
\newtheorem{defn}[thm]{Definition} 

\theoremstyle{remark} 
\newtheorem*{rem}{Remark} 

\newtheorem{obs}[thm]{Observation} 

\DeclareMathOperator{\aut}{Aut} 
\DeclareMathOperator{\out}{Out} 
\DeclareMathOperator{\inn}{Inn} 
\DeclareMathOperator{\stab}{Stab} 
\DeclareMathOperator{\fix}{Fix} 
\DeclareMathSymbol{\shortminus}{\mathbin}{AMSa}{"39} 

\tikzset{->-/.style={decoration={markings, mark=at position .55 with {\arrow{>}}}, postaction={decorate}}} 
\tikzset{-<-/.style={decoration={markings, mark=at position .55 with {\arrow{<}}}, postaction={decorate}}} 

\allowdisplaybreaks

\title{Property $R_{\infty}$ for groups with infinitely many ends}
\author{Francesco Fournier-Facio}

\address{Department of Pure Mathematics and Mathematical Statistics, University of Cambridge, UK.}
\email{ff373@cam.ac.uk}

\author{Harry Iveson}
\author{Armando Martino}

\address{Mathematical Sciences, Building 54, University of Southampton, Southampton, SO17 1BJ}
\email{H.M.J.Iveson@soton.ac.uk}
\email{A.Martino@soton.ac.uk}

\author{Wagner Sgobbi}

\address{UFES - Universidade Federal do Espírito Santo - Campus Goiabeiras}
\email{wagner.sgobbi@ufes.br}

\author{Peter Wong}
\address{Bates College, Department of Mathematics, Lewiston ME 04240, USA}
\email{pwong@bates.edu}

\date{}

\begin{document}

\begin{abstract}
    We show that an accessible group with infinitely many ends has property $R_{\infty}$. That is, it has infinitely many twisted conjugacy classes for any twisting automorphism. We deduce that having property $R_{\infty}$ is undecidable amongst finitely presented groups. 

    We also show that the same is true for a wide class of relatively hyperbolic groups, filling in some of the gaps in the literature. Specifically, we show that a non-elementary, finitely presented relatively hyperbolic group with finitely generated peripheral subgroups which are not themselves relatively hyperbolic, has property $R_{\infty}$.

\end{abstract}

	\maketitle
	
	\tableofcontents


\section{Introduction}

Property $R_{\infty}$ is a group theoretic property with connections to fixed point theory and which has been studied extensively by many authors. A non-exhaustive list might include \cite{Cox2019}, \cite{Dekimpe2014},  \cite{Goncalves2021}, \cite{Goncalves2009}, \cite{Soroko2024} and \cite{Witdouck2023}.

It is a generalisation of the property of having infinitely many conjugacy classes. Instead, one asks that there are infinitely many twisted conjugacy classes, where one twists one side of the conjugacy by an automorphism. Property $R_{\infty}$ then asks that a group have infinitely many twisted conjugacy classes for any automorphism - see Definition~\ref{def:rinf}. 


Our approach is to extend the techniques of \cite{Levitt2000}, who proved (implicitly) that any non-elementary hyperbolic group has property $R_{\infty}$. Our main theorem is that any accessible group with infinitely many ends has $R_{\infty}$.

\begin{restatable*}{thm}{mainends}
\label{inf ends has rinf} 

	Any accessible group $G$ with infinitely many ends has property $R_{\infty}$.

    In particular, any finitely presented group with infinitely many ends has property $R_{\infty}$.
\end{restatable*}

We note that, intuitively, the ends of a group is the number of components of the group at infinity. More concretely, a group with infinitely many ends acts non-trivially on a simplicial tree with finite edge stabilisers. See Theorem~\ref{stallings} for more detail, or \cite{Dicks2010} for a more extensive reference source. In particular free products of groups have infinitely many ends (except for the case of the infinite Dihedral groups) and we start by first proving the result for free products.

In fact, from the result on free products we quickly deduce that the property of having $R_{\infty}$ is undecidable amongst finitely presented groups. 

\begin{restatable*}{cor}{undecidable}
\label{R infinity undecidable}
    The property of being $R_{\infty}$ is undecidable amongst finitely presented groups. 
\end{restatable*}

At this point we should mention the result \cite[Theorem 3.3]{Felshtyn2010} where it is claimed that any non-elementary relatively hyperbolic group has property $R_{\infty}$. The proof strategy is modelled on the proof in \cite{Levitt2000}, but the technical details are all but absent and to that extent it is very hard to verify the proof given. Specifically, the technique used there is due to Paulin \cite{Paulin1991} and \cite{Paulin1997} where one takes a limit of hyperbolic spaces to produce a limiting $\R$-tree, extended by \cite{Belegradek2008} to the relatively hyperbolic case. 

Already there is an issue that the limiting tree is not projectively fixed by the automorphism, when using the results in \cite{Belegradek2008} (which is required in the proof). The result \cite[Theoreme A]{Paulin1997} does provide a projectively fixed $\R$-tree but that result is not proved in the relatively hyperbolic case and it seems that the result there would require that the peripheral subgroups be left invariant by the automorphism. In fact, this is a hypothesis that we require in our result about relatively hyperbolic groups in Theorem~\ref{rel hyp and peripheral} and we suspect that it is an essential one for these proof techniques. 

There are then various technical details missing in the treatment from \cite{Felshtyn2010}. For instance, finite generation of the parabolic subgroups is not mentioned in \cite{Felshtyn2010}, although it does appear as a hypothesis in \cite{Belegradek2008} and also appears in our Theorem~\ref{rel hyp general}. Leaving aside the invariance of the limiting tree, the next step in is then to invoke \cite[Proposition 3.2]{Levitt2000} (and also \cite[Proposition 3.1]{Levitt2000} in the case `$\lambda =1$'). 

These Propositions have various hypotheses but the main ones for \cite[Proposition 3.2 ]{Levitt2000} are that the limiting tree should be (i) Irreducible (ii) Have finite arc stabilisers and (iii) Admit finitely many orbits of directions at each branch point. 

The first of these is very plausible but merely asserted in \cite{Felshtyn2010}. For the second, smallness is invoked but without due care. Certainly, the argument given in \cite{Paulin1991} is not sufficient to produce virtually cyclic stabilisers as one needs to worry about parabolic subgroups. In fact, \cite[Theorem 1.2]{Belegradek2008} specifically proves that arc stabilisers are elementary, by which they include the parabolic subgroups. While it is possible that this technical obstacle might be overcome with a specific construction - our Lemma~\ref{trivial arc} does prove that arc stabilisers are not parabolic when one takes a limit of edge-free trees - it really merits a careful argument. In short, this hypothesis on arc stabilisers is claimed in \cite{Felshtyn2010} but is not supported by the literature.


Finally, the only time (iii) is addressed is by reference to the paper \cite{Bestvina1992a} which does not seem to be relevant to that issue (that there are finitely many orbits of directions under the stabiliser of a branch point). It is possible that the paper meant was \cite{Bestvina1991} as referenced by \cite{Levitt2000}, but this is also not a correct reference for that fact; however, in \cite{Levitt2000} there are other ways of obtaining this result which do not seem available in the relatively hyperbolic case. 

That said, the broad strokes of the argument are correct and we address these technical issues carefully to produce our main result about groups with infinitely many ends as well as a fairly general result about relatively hyperbolic groups.

\begin{restatable*}{thm}{relhyp}
\label{rel hyp general}
	Let $G$ be a non-elementary finitely presented relatively hyperbolic group $G$ whose peripheral subgroups are finitely generated but not relatively hyperbolic. Then $G$ has property $R_{\infty}$. 
\end{restatable*}

\begin{rem}
    We note that this Theorem recovers the result of \cite{Levitt2000} that any non-elementary hyperbolic group has property $R_{\infty}$, although we really use their argument except for the fact that we make use of JSJ decompositions at the final stage.  
\end{rem}

Our proof uses train track methods for free products so that we can refer to the literature to address the technical issues mentioned above. We then show that groups with infinitely many ends have $R_{\infty}$ by extending the results for free products using the tree of cyclinders construction of Guirardel and Levitt, \cite{Guirardel2007} and \cite{Guirardel2011}. The idea here is that a group with infinitely many ends 
either admits a finite normal subgroup so that the quotient is a free product or admits a tree which is invariant under all automorphisms. See Theorem~\ref{infinitely many ends}. 

Finally, we prove the Theorem for relatively hyperbolic groups using the canonical JSJ splittings of \cite{Guirardel2016}. 

\medspace



The first version of this article was placed on the arXiv by HI, AM, WS and PW. We then received communications from FFF and Anthony Genevois independently observing that it is possible to broaden the results in our paper. We are hugely grateful for these observations and insights. 

The most general way to do this seems to be via quasi-morphisms, and this is detailed in Appendix B. The most general statement of these results is as follows.

\begin{restatable*}{cor}{generalrinfty}
	\label{generalrinfty}
	Let $G$ be a finitely generated group satisfying one of the following properties.
	\begin{enumerate}
		\item $G$ is a non-elementary hyperbolic group.
		\item $G$ is non-elementary hyperbolic relative to a collection of finitely generated subgroups, none of which is relatively hyperbolic.
		\item $G$ has infinitely many ends.
		\item $G$ is a graph product of groups over a finite graph that does not decompose non-trivially as a join and is not reduced to a single vertex.
		\item $G$ is a graph product of abelian groups and it is not virtually abelian.
	\end{enumerate}
	Then $G$ has property $R_\infty$.
\end{restatable*}

We note that the main body of this paper constructs translation length functions arising from trees which are constant on twisted conjugacy classes, whereas the method via quasi-morphisms construct such functions which are bounded on twisted conjugacy classes (albeit for all automorphisms at the same time). Nevertheless, the methods should be seen as analogous both in their overall strategy and in the details of their technical construction. 

\subsection*{Acknowledgements}

FFF is supported by the Herchel Smith Postdoctoral Fellowship Fund. He thanks Yuri Santos Rego for useful comments. WS was financed, in part, by the São Paulo Research Foundation (FAPESP), Brasil. Process Numbers 2017/21208-0 and 2019/03150-0.


\section{Twisted Conjugacy and the  \texorpdfstring{$R_{\infty}$ } \  Property}

\subsection{Twisted Conjugacy}			

\begin{defn}\label{defn twisted conjugate}
	Let $G$ be a group and $\varphi\in\aut(G)$ an automorphism of $G$.
	We define a relation $\sim_{\varphi}$ on $G$ by $x\sim_{\varphi}y$ if and only if $\exists w\in G$ with $(w\varphi)xw^{-1}=y$.
	If $x\sim_{\varphi}y$, we say that $x$ and $y$ are {twisted conjugates} in $G$.
\end{defn}

\begin{rem}
    We take the action of $\aut(G)$ on $G$ to be a right action, writing $g\cdot\varphi$ or $g\varphi$ for the image of $g\in G$ under the automorphism $\varphi\in\aut(G)$.
\end{rem}

\begin{rem}
    Note that there is an asymmetry in the definition of twisted conjugacy. The side on which we place the inverse makes no difference, but the side on which we put the automorphism does. However, 
    $$
    w x (w \varphi)^{-1} = (u \varphi^{-1}) x u^{-1}, \text{ where } u = w \varphi.
    $$

    This shows the `left' and `right' versions of twisted conjugacy are related at the cost of changing the automorphism to its inverse. Since the property of having $R_{\infty}$ - Definition~\ref{def:rinf for autos} - is about all possible automorphisms, it makes no difference which we choose. 
\end{rem}

The following is a standard fact and an easy exercise.

\begin{lemma}
	Twisted conjugacy ($\sim_{\varphi}$) is an equivalence relation on $G$.
\end{lemma}


\begin{defn}[$R_{\infty}$]
\label{def:rinf for autos}
	Let $G$ be a group and $\varphi\in\aut(G)$.
	We say that $G$ has property $R_{\infty}(\varphi)$ if $\sim_{\varphi}$ has infinitely many equivalence classes in $G$.
	We say that $G$ has property $R_{\infty}$ if $G$ has $R_{\infty}(\varphi)$ for all $\varphi\in\aut(G)$.
\end{defn}

\begin{lemma}\label{lemma inner autos preserve twisted conjugacy}
	If $\iota\in\inn(G)$ then $G$ has $R_{\infty}(\varphi)$ if and only if $G$ has $R_{\infty}(\varphi\iota)$.
\end{lemma}

\begin{proof}
	Let $x\in G$ and let $\iota_{x}\in\inn(G)$ be the automorphism $g\mapsto x^{-1}gx=g^{x}$.
	Let $y,z\in G$. Then:
	\begin{align*}
		&	\ y \sim_{\varphi} z	\\
		\Leftrightarrow	&	\ \exists w\in G \text{ such that } z=(w\varphi)yw^{-1}	\\
		\Leftrightarrow	&	\ \exists w\in G \text{ such that } xz=x(w\varphi)(x^{-1}x)yw^{-1}=(w\varphi)^{x}(xy)w^{-1}=(w\varphi\iota_{x})(xy)w^{-1}	\\
		\Leftrightarrow	&	\ xy \sim_{\varphi\iota_{x}} xz .
	\end{align*}
	Thus we have a bijection between the equivalence classes of $\sim_{\varphi}$ and those of $\sim_{\varphi\iota_{x}}$.
	Hence $G$ has $R_{\infty}(\varphi)$ if and only if $G$ has $R_{\infty}(\varphi\iota_{x})$.
\end{proof}

\begin{defn}
\label{def:rinf}
	Let $G$ be a group and let $\Phi\in\out(G)$.
	We say that $G$ has property $R_{\infty}(\Phi)$ if $G$ has $R_{\infty}(\varphi)$ for some (and hence every) automorphism $\varphi\in\Phi$.
\end{defn}

\begin{rem}
	By Lemma \ref{lemma inner autos preserve twisted conjugacy}, this concept of $R_{\infty}(\Phi)$ for $\Phi\in\out(G)$ is well-defined, and it follows that $G$ has $R_{\infty}$ (i.e. $G$ has $R_{\infty}(\varphi)$ for all $\varphi\in\aut(G)$) if and only if $G$ has $R_{\infty}(\Phi)$ for all $\Phi\in\out(G)$.
\end{rem}

\begin{rem}
As above, let $\Phi\in\out(G)$ and fix $\varphi \in \Phi$. Following \cite{Levitt2000}, we say that $\alpha,\beta \in \Phi$ are isogredient if and only if there exists $g \in G$ such that $\alpha=\iota_g \beta \iota_g^{-1}$, where $\iota_g$ denotes the inner automorphism with $\iota_g(h)=ghg^{-1}$ for all $h \in G$. We denote by $S(\Phi)$ the set of equivalence classes (called isogredience classes) of automorphisms representing $\Phi$. We say $G$ has property $S_\infty$ if $S(\Phi)$ is infinite for every $\Phi \in \out(G)$. As stated in the third paragraph of our introduction, we attribute the proof of property $R_\infty$ for non-elementary hyperbolic groups primarily to \cite{Levitt2000}. Indeed, the authors show in Theorem~3.5 that any such group possesses property $S_\infty$. But this, in turn, easily implies property $R_\infty$, as we show below. This observation was also made by the paper \cite{Felshtyn2004}. Write $\alpha=\iota_r \varphi$ and $\beta=\iota_s \varphi$ for $r,s \in G$. We claim that $\alpha$ and $\beta$ are isogredient if and only if $\overline{r}\sim_{\overline{\varphi}}\overline{s}$ in the quotient group $G/Z(G)$, where $\overline{\varphi}$ is the automorphism naturally induced from $\varphi$. In fact, $\alpha$ and $\beta$ being isogredient means $\iota_r\varphi=\iota_g\iota_s\varphi\iota_g^{-1}$ for some $g \in G$. Since $\varphi\iota_g^{-1}=\iota_{\varphi(g)^{-1}}\varphi$, the former equation is equivalent to $\iota_{r}=\iota_{gs\varphi(g)^{-1}}$, or $gs\varphi(g)^{-1}r^{-1} \in Z(G)$, which in turn is clearly equivalent to saying that $\overline{r}\sim_{\overline{\varphi}}\overline{s}$ in $G/Z(G)$. Therefore, if the set $S(\Phi)$ is infinite, then $G/Z(G)$ has $R_\infty(\overline{\varphi})$, which implies $G$ has $R_\infty(\varphi)$ (see Lemma \ref{lemma G/N R-infiity implies G R-infinity}). We conclude that property $S_\infty$ for $G$ implies property $R_\infty$ for $G$.
\end{rem}

\subsection{Mapping Torus} 			

Our argument about twisted conjugacy classes uses the standard technique of converting questions about twisted conjugacy in a group $G$ to ones concerning genuine conjugacy in a mapping torus of $G$. We therefore recall the definition of a mapping torus.

\begin{defn}[Mapping torus]
	Let $G$ be a group and $\varphi\in\aut(G)$.
	The {mapping torus} of $\varphi$ is the semi-direct product $M_{\varphi}:=G\rtimes_{\varphi}\mathbb{Z}=\langle G,t | g^{t}=g\varphi \ \forall g\in G \rangle$.
\end{defn}

\begin{rem}
Since $M_{\varphi}$ is a semi-direct product, elements of $M_{\varphi}$ have a standard form $gt^{k}$ where $g\in G$ and $k\in\mathbb{Z}$ are unique.
	Note that for any $h\in G$ we have $t^{-1}h=h^{t}t^{-1}=(h\varphi)t^{-1}$, thus an element $t^{k}h$ can be written in the alternate standard form $(h\cdot\varphi^{-k})t^{k}$.
\end{rem}

A key observation is that twisted conjugacy is realised as standard conjugacy in the mapping torus and that having $R_{\infty}(\varphi)$ amounts to there being infinitely many conjugacy classes of a certain type in the mapping torus. 

\begin{lemma}
	\label{conjclassesofmappingtorus}
	Let $G$ be a group and $\varphi\in\aut(G)$.
	Then $G$ has $R_{\infty}(\varphi)$ if and only if the set $\{ tx \ | \ x\in G \}$ has infinitely many $M_{\varphi}$ conjugacy classes.
\end{lemma}

\begin{proof}
	Let $x,y\in G$.
	We have that $tx$ and $ty$ are conjugate in $M_{\varphi}$ if and only if there exists $u=gt^{k}\in M_{\varphi}$ so that $ty=(tx)^{u}=u^{-1}(tx)u$.
	Observe that given $x\in G$ we can always find $h\in G$ so that $u=gt^{k}=(tx)^{k}h$.
	Then $(tx)^{u}=(tx)^{(tx)^{k}h}=(tx)^{h}$.
	Now for any $x,y\in G$, we have:
	\begin{align*}
		&	tx \text{ and } ty \text{ are conjugate in } M_{\varphi}	\\
		\Longleftrightarrow 	&	\exists h\in G \text{ with } ty=(tx)^{h}=h^{-1}txh=t(h^{-1}\varphi)xh	\\
		\Longleftrightarrow 	&	\exists \hat{h}\in G \text{ with } y=(\hat{h}\varphi)x\hat{h}^{-1}	\\
		\Longleftrightarrow	&	y\sim_{\varphi} x 	.
	\end{align*}
\end{proof}


\section{Trees and Group Actions}

\subsection{Simplicial trees and  \texorpdfstring{$\R$-\text{trees}}  \  }				

We refer the reader to \cite{Culler1987} for a treatment on $\R$-trees and to \cite{Chiswell2001} for a more general textbook on $\Lambda$-trees. The definitions and results here are mainly based on \cite{Culler1987} and \cite{Chiswell2001}. We note that in the case $\Lambda= \R$ these concepts coincide, whereas the case $\Lambda = \Z$ corresponds to the case of a simplicial tree below.  

\begin{defn}
	A {simplicial graph}, $\Gamma$, is a 4-tuple, $(V, E, \sigma, \tau)$ where:
	\begin{itemize}
		\item $V$ is a set called the {vertices} of $T$,
		\item $E \subseteq V \times V$ is the set of {oriented edges} of the graph,
		\item $\sigma: E \to V$ and $\tau : E \to V$ are {incidence maps}, defined by $\sigma(u,v) = u$ and $\tau(u,v) = v$.
	\end{itemize}
	Moreover, for every $e = (u,v)\in E$  we always have that $(v,u) \in E$; we call this the {inverse edge}, denoted, $\overline{e}$. 
\end{defn}

\begin{defn}
	An {edge path} in a graph, $\Gamma =(V, E, \sigma, \tau) $ is a sequence of vertices, $v_0, v_1, \ldots, v_k$ where for each $0 \leq i \leq k-1$, $(v_i, v_{i+1}) \in E$. We allow the edge path to consist of a single vertex, $v_0$. 
	
	\begin{itemize}
		\item The edge path is called {trivial} when it consists of a single vertex, and {non-trivial} otherwise.
		\item An edge path $v_0, v_1, \ldots, v_k$ is said to start at $v_0$ and end at $v_k$.
		\item A non-trivial edge path may also be described as a sequence of edges, $e_0 \ldots e_{k}$, where $\tau(e_i) = \sigma(e_{i+1})$ for $0 \leq i \leq k-1$. 
		\item An edge path, $e_0 \ldots e_{k}$ is called {reduced} if, for all $0 \leq i \leq k-1$, $e_i \neq  \overline{{e_{i+1}}}$.  A trivial path is always considered reduced.
	\end{itemize}
\end{defn}

\begin{defn}
	A {simplicial tree} is a simplicial graph where between any two vertices there is a unique reduced edge path starting at one and ending at the other. 
    
    \noindent
    For a tree, we let $[u,v]$ denote the unique reduced edge path from $u$ to $v$; this is called the segment from $u$ to $v$. 
\end{defn}

\begin{defn}[Culler--Morgan \cite{Culler1987}]\label{defn R-tree}
	An {$\mathbb{R}$-tree} $T$ is a path-connected non-empty metric space so that for any points $x,y\in T$, there is a unique arc $[x,y]\subseteq T$ joining $x$ and $y$, which is isometric to the interval $[0,d(x,y)]\subseteq\mathbb{R}$ (where $d$ is the metric on $T$).
	
	Equivalently, an $\mathbb{R}$-tree is a 0-hyperbolic geodesic metric space, in the sense of Gromov.
\end{defn}

\begin{defn}
	Given two points, $u,v$ in an $\R$-tree $T$, we denote by $[u,v]$ the unique geodesic from $u$ to $v$ in $T$. This is called the {segment} from $u$ to $v$.
\end{defn}

\begin{defn}\label{defn directions and  branch points}
	Given an $\mathbb{R}$-tree $T$ and a point $x\in T$, a {direction} at $x$ is a connected component of $T-\{x\}$.
	We say a point $x\in T$ is a {branch point} if there are at least three directions at $x$.
\end{defn} 

\begin{defn}
	
	\ 
	
	\begin{itemize}
		\item We say that an $\R$-tree is a {metric simplicial tree} (or simply a simplicial $\R$-tree) if the set of branch points is a discrete subset of the tree. 
		\item If $T$ is a simplicial $\R$-tree and $V$ a discrete subset of $T$ which includes all the branch points (but may include more points), then the edges are all the segments $[u,v]$ between elements of $V$ where $[u,v] \cap V = \{ u,v\}$. Any non-trivial segment between vertices may then be given as an edge-path, $e_1 \ldots e_k$ where each $e_i$ is an edge. 
		\item Conversely, a simplicial tree $T$ given as a set of vertices and edges may be made into a simplicial $\R$-tree by assigning a positive length to each edge and making $T$ into a metric space via the corresponding path metric. 
	\end{itemize}
\end{defn}

\begin{defn}
	Let $T$ be an $\R$-tree. Then for any three points, $u,v,w \in T$, the $Y$-point, $Y(u,v,w)=y \in T$, is given by: 
	$$
	[u,v] \cap [u,w] = [u,y]. 
	$$
\end{defn}

\begin{lemma}[ {\cite[Chapter 2, Lemma 1.2]{Chiswell2001}}]
	Let $T$ be an $\R$-tree. Then for any three points, $u,v,w \in T$, the $Y$-point is unique and does not depend on the order of the points. 
	Moreover, $Y(u,v,w)$ is the point on $[u,v]$ whose distance from $u$ is given by the Gromov product, $(v.w)_u = \frac{1}{2}(d_T(u,v) + d_T(u,w) - d_T(v,w))$. 
\end{lemma}

\subsection{Group actions on  \texorpdfstring{$\R$} \ -trees}			

Throughout this subsection, we will consider a group, $G$, acting isometrically on an $\R$-tree, $T$. We note that if one starts with a simplicial tree and a group action sending vertices to vertices and edges to edges then the process of making this a simplicial $\R$-tree described above allows one to extend the group action to an isometric action; that is, a group of automorphisms of a tree preserves the induced path metric. 

\begin{defn} [{\cite[p. 576 and Definition 1.4]{Culler1987}}] \label{defn translation length}
Suppose that $G$ is a group acting isometrically on an $\R$-tree, $T$. 
	\begin{enumerate}[(i)]
		\item For any $g \in G$ we define the {translation length} of $g$ (with respect to $T$) to be, 
		$$
		\| g\|_T := \inf_{x \in T} \{ d_T(x,xg)  \}.
		$$
		We write $\|g\|$ for $\|g\|_T$ if $T$ is understood. 
		\item Define the characteristic set of $g$ to be, 
		$$
		A_g = \{ x \in T \  : \ d_T(x, xg) = \| g \| \}. 
		$$
		\item $g \in G$ is called ellipic if $\|g\|_T=0$ and hyperbolic if $\|g\|_{T} > 0$. 
	\end{enumerate}
\end{defn}

\begin{lemma}[{\cite[Lemma 1.3 (p.576)]{Culler1987}}]
	\label{treefacts}	Let $G$ act isometrically on an $\R$-tree, $T$ and let $g \in G$. 
	
	\begin{enumerate}[(i)]
		\item There exists an $x \in T$ such that $d_T(x,xg) = \| g \|$. That is, the infimum in Definition \ref{defn translation length}(i) is a minimum . 
		\item The set $A_g = \{ x \in T \ :d_T(x, xg) = \| g \| \ \}$ is non-empty (by the previous part). It is also a closed subtree of $T$, invariant under the action of $g$. 
		\item If $g$ is elliptic, then $A_g$ is the fixed point set  of $g$, $\fix(g)$. 
		\item If $g$ is hyperbolic, then $A_g$ is called the axis of $g$ and is isometric to the real line. It is the smallest $g$-invariant subtree of $T$. The element $g$ acts on $A_g$ as a translation by the real number $\|g \|$. 
        \item If $g$ is hyperbolic and $0 \neq n \in \Z$, then $A_{g^n} = A_g$ and $\| g^n \| = |n| \|g\|$. 
	\end{enumerate}
\end{lemma}

\begin{lemma}\label{treefacts2}
Let $G$ act isometrically on an $\R$-tree $T$. Let $g \in G$ and $x \in T$ any point.
    \begin{enumerate}[(i)]
        \item The midpoint of the segment $[x,xg]$ lies in $A_g$. 
		\item For any hyperbolic $g \in G$, we have $Y(xg^{-1}, x, xg) \in A_g$.  
    \end{enumerate}
\end{lemma}

\begin{proof}
    \begin{enumerate}[(i)]
        \item This is just \cite[Chapter 3 (Lemma 1.1 for the elliptic case and Theorem 1.4 for the hyperbolic case)]{Chiswell2001} . 
		\item This is \cite[Chapter 3, Theorem 1.4]{Chiswell2001}.
    \end{enumerate}
\end{proof}

\begin{obs}
\label{commute}
    It follows easily that $A_{g^h} = A_g \cdot h$. Hence if $g$ and $h$ commute then $h$ preserves $A_g$. Further, if $g$ and $h$ commute and are both hyperbolic then $A_g=A_h$, since $h$ preserves the line $A_g$ but $A_h$ is the smallest $h$-invariant subtree of $T$. 
\end{obs}

\begin{defn}
	\label{length function}
	Let $G$ act isometrically on an $\R$-tree, $T$. Then the length function of this action is the function, $l: G \to \R$ given by $l(g) = \| g \|_T$. This is also called the translation length function. 
\end{defn}

\begin{defn}[Culler--Morgan \cite{Culler1987}]

\label{properties of R-trees}
	Let $T$ be an $\mathbb{R}$-tree equipped with an action of a group $G$ by isometries. We say that $T$ is:
	\begin{itemize}
		\item {irreducible}, if there is no point, line or end of $T$ which is invariant under the action of $G$. Equivalently, there exist a pair of groups elements, $g, h \in G$ which are hyperbolic and whose axes meet in an arc of finite positive length (See \cite[Theorem 2.7]{Culler1987} for this equivalence).
		\item {minimal}, if there is no proper $G$-invariant subtree of $T$.
		\item {non-trivial}, if there is no global fixed point (i.e. no $x\in T$ such that $x\cdot G=x$).
		\item {small}, if for any non-trivial arc $[x,y]$, the pointwise stabiliser $\stab([x,y]) = \{ g \in G \ : \ zg = z, \  \forall z \in [x,y]\} \le G$ does not contain a free subgroup of rank 2. 
	\end{itemize}
\end{defn}
\begin{rem}
	It is also true that the action of $G$ on $T$ is irreducible if and only if there are a pair of isometries whose axes are disjoint \cite[Lemmas 2.1 and 1.5 and 1.6]{Culler1987}. Concretely, if $g, h$ are hyperbolic and $A_g \cap A_h$ meets in a segment of finite poisitive length then, for some large $n \in \Z$, the axes of $g$ and $g^{h^n}$ will be disjoint.  
\end{rem}

We shall also need the following. 

\begin{lemma}[Culler-Morgan \cite{Culler1987}, 1.5 and 1.6]
\label{disjoint axes}
    Let $g,h$ be hyperbolic isometries of an $\R$-tree $T$ whose axes $A_g, A_h$ are disjoint. Then $\|gh\|= \|hg \|=\|g\|+\|h\| + 2d(A_g,A_h)$. In particular $gh$ and $hg$ are hyperbolic. 
\end{lemma}

\begin{prop}[{\cite[Proposition 3.1]{Culler1987}}]
\label{minimal subtree}
   If $G$ acts on an $\R$-tree $T$ and some element $g \in G$ acts hyperbolically, then $T$ admits a unique, minimal $G$-invariant subtree. This subtree is exactly the union of all the hyperbolic axes of  elements of $G$. 
\end{prop}

\begin{lemma}[Serre {\cite[Proposition 25 (p.63)]{Serre1980}}]
	\label{Serre lemma}
    If $G$ is finitely generated and acts isometrically on an $\R$-tree, $T$, then the action is non-trivial if and only if there exists a hyperbolic element in $G$. That is, if there exists some $g \in G$ with $\|g\|_{T} > 0$. 
\end{lemma}

\begin{rem}
	The result in \cite{Serre1980} concerns simplicial trees, but the same proof works for $\R$-trees without change.
\end{rem}

\begin{defn}
	Let $T$ be a tree. For points $x_1, \ldots, x_n$ we write $[x_1, \ldots, x_n]$ to mean that the unique segment from $x_1$ to $x_n$ crosses the points $x_2, \ldots, x_{n-1}$ in the order given. 
\end{defn}

\begin{lemma}[Circle-Dot Lemma]
	\label{circledot}
	Let $G$ act isometrically on an $\R$-tree, $T$. Suppose we have distinct points, $p,q \in T$ such that $[p,q, pg, qg]$. That is, the segment from $p$ to $qg$ crosses the points $q$ and $pg$ in that order. 

	Then $g$ is hyperbolic and both $p$ and $q$ belong to the axis of $g$. In particular, $\|g\|_{T} = d(p, pg)$.  
\end{lemma}
\begin{rem}
	We allow the possibility that $q=pg$ in this Lemma.
\end{rem}
\begin{proof}
	We first show that $g$ is not elliptic. We argue by contradiction; if $g$ is elliptic then the midpoint of $[p, pg]$ is fixed by $g$ by Lemmas \ref{treefacts}(iii) and \ref{treefacts2}(i). Call this point $w$. Then $d_T(w, q) = d_T(w, qg)$. Since $w \in [p,pg]$, this forces $w \in [q, qg]$. But now $w$ is both the midpoint of $[p,pg]$ and $[q, qg]$ which is impossible if $p \neq q$. Hence $g$ is hyperbolic. 
    Therefore, by Lemma~\ref{treefacts2}(ii) it is enough to show that $p = Y(pg^{-1}, p , pg)$. 
    Indeed, since $[p,q,pg,qg]$ then $[pg^{-1},qg^{-1},p,q]$, and hence $[pg^{-1}, qg^{-1}, p, q, pg, qg]$. That is, $p\in[pg^{-1},pg]$ and therefore $p = Y(pg^{-1}, p , pg)$, as required. 
    
\end{proof}

	\begin{figure}
		\centering
		\begin{minipage}[t]{0.4\textwidth}
			\centering
			\begin{tikzpicture}[scale=0.8]
				\draw[->] (0,0) -- (8,0) node[anchor=north west, below right] {$A_g$};
				
				\draw (1,0) circle (3pt) node[anchor=south] {$p$}; 
				\fill (2.5,0) circle (3pt) node[anchor=south] {$q$}; 
				\draw (6,0) circle (3pt) node[anchor=south] {$pg$}; 
				\fill (7.5,0) circle (3pt) node[anchor=south] {$qg$}; 
			\end{tikzpicture}
			 \caption*{On the axis: \\ $\circ -  \bullet - \circ - \bullet$}
		\end{minipage}
		\hspace{2cm} 
		\begin{minipage}[t]{0.4\textwidth}
			\centering
			\begin{tikzpicture}[scale=0.8]
				\draw[->] (0,0) -- (8,0) node[anchor=north west, below right] {$A_g$};
				
				\draw (2,0) -- (2,4);
				\draw (2,3) circle (3pt) node[anchor=west] {$p$}; 
				\fill (2,1) circle (3pt) node[anchor=west] {$q$}; 
				
				\draw (6,0) -- (6,4);
				\draw (6,3) circle (3pt) node[anchor=west] {$pg$}; 
				\fill (6,1) circle (3pt) node[anchor=west] {$qg$}; 
			\end{tikzpicture}
			\caption*{Not on the axis: \\ $\circ -  \bullet - \bullet - \circ$}
		\end{minipage}
		\caption*{The Circle-Dot Lemma}
	\end{figure}

We will also need the following, 

\begin{lemma}[Paulin's Lemma {\cite[ Lemma 4.3]{Paulin1989}}]
\label{Paulins lemma}
	Let $G$ act minimally, non-trivially, isometrically and irreducibly on an $\R$-tree, $T$. Then for any $a, b \in T$ there exists a hyperbolic element $g \in G$ whose axis contains the segment $[a,b]$. 
\end{lemma}

This allows us to deduce the $R_{\infty}$ property when we have an action of the mapping torus on a tree. Thus the following Lemmas will be a key tool for deducing property $R_{\infty}$. We note that the existence of the action required by this Lemma will be discussed in Corollary~\ref{cor action of Mphi}. 

\begin{lemma}
\label{irreducible for G}
    Let $G$ be a group and let $\varphi\in\aut(G)$. If $M_{\varphi}$ acts on a tree $T$ minimally, isometrically, irreducibly then $G$ also acts minimally and irreducibly on $T$. 
\end{lemma}
\begin{proof}
    We can restrict the action to $G$ and get an isometric action of $G$ on $T$. We first claim that $G$ admits a hyperbolic element with respect to this action. 

    Since $M_{\varphi}$ acts irreducibly on $T$ we get that (by the remark after Definition~\ref{properties of R-trees}) we have two hyperbolic isometries, $m_1, m_2 \in M_{\varphi}$ whose axes are disjoint. If either of these are in $G$, then we have a hyperbolic isometry in $G$. Otherwise, we may write $m_1=t^{a} u, m_2=t^b v$, where $0 \neq a,b \in \Z$ and $u,v \in G$. Now consider, $m_1^b = t^{ab} g, m_2^{-a} = t^{-ab} h$ for some $g, h \in G$. By Lemma~\ref{treefacts}(iv) the axes of $m_1$ and $m_1^b$ are equal, as are the axes of $m_2$ and $m_2^{-a}$. Hence the axes of $t^{ab} g$ and $t^{-ab} h$ are disjoint. Therefore, by Lemma~\ref{disjoint axes}, $(t^{ab} g) (t^{-ab} h) \in G$ is hyperbolic. This proves the claim. 

    Now, by Lemma~\ref{minimal subtree}, $G$ admits a minimal $G$-invariant subtree. But since $G$ is a normal subgroup of $M_{\varphi}$, this is also a $M_{\varphi}$-invariant subtree and so is the whole of $T$, as $M_{\varphi}$ acts minimally. This proves that $G$ also acts minimally on $T$. 

    Finally, to prove that $G$ acts irreducibly on $T$ notice that if there were an invariant line, then $T$ would have to be a line by minimality, contradicting the fact that $M_{\varphi}$ acts minimally and irreducibly. If $G$ were to admit a fixed end, then any hyperbolic axes $A_g$, $A_h$ of hyperbolic elements of $G$ would intersect in this end. However, if $g \in G$ is hyperbolic and $m, \in M_{\varphi}$, then $g^m \in G$ is also hyperbolic and $A_{g^m} = (A_g)m$, showing that the end would need to be invariant by all elements of $M_{\varphi}$, again contradicting the fact that $M_{\varphi}$ acts irreducibly.

\end{proof}

\begin{lemma}\label{lemma Mphi acting on T}
	Let $G$ be a group and let $\varphi\in\aut(G)$. If $M_{\varphi}$ acts on a tree $T$ minimally, isometrically, irreducibly, then $G$ has $R_{\infty}(\varphi)$.
\end{lemma}

\begin{proof}
 Notice that by Lemma~\ref{irreducible for G}, $G$ also acts on $T$ minimally and irreducibly. In particular, Paulin's Lemma~\ref{Paulins lemma} applies to the action of $G$ on $T$.

	We will prove the Lemma by showing that we have infinitely many conjugacy classes of the form $tx$, where $x \in G$, and deduce the result by Lemma~\ref{conjclassesofmappingtorus}. 
	
    \medskip
    
	\noindent
	Case (i): First suppose that $t$ acts hyperbolically on $T$. Choose $a$ in the axis of $t$ and let $b=at^2$. Then, by Paulin's Lemma, there exists a hyperbolic $g \in G$ whose axis contains the segment, $[a,b]$. 
	
	As the axes of $t$ and $g$ intersect non-trivially, we can assume (by replacing $g$ with $g^{-1}$ if needed) that $t$ and $g$ translate in the same direction along their intersection. In that case, using the Circle-Dot Lemma (Lemma~\ref{circledot}) with $p=a, q=at$, we have that $\|tg^n\|= \|t\| + n \|g\|$ for all positive integers $n$.
    Hence these elements are all in distinct conjugacy classes, since translation length is a conjugacy invariant. 
	
    \medskip
	
	\noindent
	Case(ii)(a): Next suppose that $t$ is elliptic and that $|\fix(t)| > 1$. Here choose $a \neq b \in \fix(t)$ and let $g \in G$ be a hyperbolic element whose axis contains $[a,b]$. By replacing $g$ with a sufficiently large (but possibly negative) power, we may assume that $b \in [a, ag]$.

	It is straightforward to verify using the Circle-Dot Lemma (Lemma~\ref{circledot}) with $p=a, q=b$ that $\|t g^n\| = n\|g\|$ for all positive integers $n$ and hence, again by Lemma~\ref{conjclassesofmappingtorus}, we are done.

    \medskip
	
	\noindent
	Case(ii)(b): Lastly suppose that $t$ is elliptic and fixes a unique point in the tree. Since the action is $G$-irreducible, $G$ admits two hyperbolic elements whose axes are disjoint. In particular, we may find a hyperbolic $g \in G$ whose axis does not contain the unique fixed point of $T$.

	Let $p$ be the fixed point of $t$ and let $q$ be the point on the axis of $g$ closest to $p$. The the Circle-Dot Lemma~(\ref{circledot}) gives us that $\|t g^n \| = 2d(p,q) + |n|\| g \|$. Hence these elements fall into infinitely many conjugacy classes.

\end{proof}

\begin{rem}
    Note that when using Lemma~\ref{lemma Mphi acting on T} we refer to Lemma~\ref{irreducible for G} to deduce that the action of $G$ on the same tree is also irreducible and minimal. In practice, however, we will start with a minimal irreducible action of $G$ and extend to one on the mapping torus via Corollary~\ref{cor action of Mphi}. In that sense, Lemma~\ref{irreducible for G} is redundant but we record it here for completeness. 
\end{rem}

\subsection{Deformation spaces and reduced trees}

In this section we wish to explore the concepts of a `reduced' simplicial tree, which is different to that of an irreducible tree. 

We start with the definition of a deformation space. This discussion is taken from \cite{Guirardel2007}, to which we refer the reader for a fuller account. 


\begin{defn}[{\cite[Definition 3.2]{Guirardel2007}}] Let $T,S$ be two simplicial $G$-trees. 
\begin{itemize}
    \item     We say that $T$ and $S$ have the same elliptic subgroups if for every subgroup $H$ of $G$ which fixes a point of $T$, $H$ also fixes a point of $S$, and vice versa. 
    \item $T$ and $S$ are then said to be in the same deformation space if they have the same elliptic subgroups. 
    \item 
The deformation space $\mathcal{D}$ containing $T$ is then the set of  all $G$-trees (up to equivariant isometry) which have the same elliptic subgroups as $T$.
\end{itemize}

\end{defn}

\begin{rem}
    Note that the notion of equivariant isometry makes sense since we can regard our simplicial trees $T$ as simplicial $\R$-trees by assiging length 1 to each edge. 
\end{rem}

\begin{defn}[{\cite[Definition 3.12]{Guirardel2007}}]
Given a deformation space $\mathcal{D}$ and a collection of subgroups $\mathcal{A}$ of $G$ (closed under conjugacy and passing to subgroups) we define $\mathcal{D_A} \subseteq \mathcal{D}$ as the set
of $G$-trees $T \in \mathcal{D}$ whose edge stabilizers belong to $\mathcal{A}$. We call $\mathcal{D_A}$ a restricted deformation space, or simply a deformation space.
\end{defn}

\begin{prop}[{\cite[Proposition 3.10 (1)]{Guirardel2007}}]
Let $\mathcal{D}$ be a deformation space. If one tree $T \in \mathcal{D}$ is irreducible, then all trees in $\mathcal{D}$ are irreducible. 
\end{prop}

\begin{defn}[{\cite[Definition 3.5]{Guirardel2007}}]
\label{definition of reduced}
	A simplicial $G$-tree $T$ is called reduced if whenever $e=(u,v)$ is an edge of $T$ with $\stab_G(e) = \stab_G(v)$, then $u$ and $v$ are in the same $G$-orbit. 
\end{defn}

\begin{obs}
    Note that any deformation space contains a reduced tree. See \cite{Guirardel2007}, Definitions 3.3, 3.4 and 3.5. 
\end{obs}

\begin{lemma}
\label{reduced implies minimal}
	Let $T$ be a simplicial reduced $G$-tree. Then $T$ is also minimal. 
\end{lemma}
\begin{proof}
We prove the contrapositive. So suppose that $T$ is not minimal; we will deduce that it is not reduced. 

We consider some proper, $G$-invariant subtree, $S$. Let $e=(u,v)$ be an edge of $T$ such that $u \in S$ but $v \not\in S$. Then $v \not\in S \supseteq u.G$, as $S$ is $G$-invariant and so $u$ and $v$ are not in the same $G$ orbit. It then suffices to show that $\stab_G(v) \subseteq \stab_G(e)$ (as we always have that $\stab_G(e) \subseteq \stab_G(v)$).

Consider $g \in \stab_G(v)$. Then the edge path, $e,\overline{e}g$ starts at $u$ and ends at $ug$. Since neither $e$ nor $eg$ is an edge of $S$ and $u, ug \in S$ are connected by a path in $S$, we must have that this path fails to be reduced (and $u=ug$). Thus $e=eg$, as required. 
\end{proof}

\begin{obs}[\cite{Guirardel2007}]
\label{reduced tree facts}
    Let $G$ be a finitely generated group acting on a simplicial tree, $T$. 
\begin{enumerate}[(i)]
\item An edge $e$ of $T$ is called collapsible if collapsing $e$ to a point produces a $G$-tree in the same deformation space as $T$, \cite[Definition 3.3]{Guirardel2007}. 
    \item We can collapse all collapsible edges of $T$ to produced a reduced $G$-tree, $r(T)$, the reduction of $T$. Any edge stabiliser of $r(T)$ stabilises an edge of $T$, \cite[Definition 3.5]{Guirardel2007}. 
    \item $T$ and $r(T)$ lie in the same deformation space. In particular, the action on $T$ is non-trivial if and only if $r(T)$ is not a point, \cite[Theorem 3.8]{Guirardel2007} and Lemma~\ref{reduced implies minimal}. 
    \item If $T$ is reduced and the edge stabilisers of $T$ are finite, then all reduced trees in the same deformation space as $T$ have the same edge stabilisers as $T$, \cite[Proposition 7.1 (3) and Corollary 7.3]{Guirardel2007}
\end{enumerate}
\end{obs}



\subsection{Irreducible actions on trees}

We write down a fairly simple criterion for an action of a group on a tree to be irreducible. This will be useful in what follows. 

First we will need the following Lemma.

\begin{lemma}
\label{virtually abelian}
   Suppose  $G$ is a finitely generated group and the derived subgroup $[G,G]$ is finite.  Then G is virtually abelian.
\end{lemma}
\begin{proof}
    Any element g of G has finitely many conjugates since  $x^{-1} g x = g g^{-1} x^{-1} gx \in g[G,G]$.  Therefore, the centraliser of any element has finite index. Hence the intersection of all the centralisers of the generators has finite index, as G is finitely generated. But this intersection is clearly central in G and hence abelian. Therefore G is virtually abelian. 
\end{proof}

\begin{prop}
\label{irreducible finite edge stabs}
Let $G$ be a finitely generated group acting minimally and non-trivially on a tree, $T$. Suppose further that the action is reducible. Then, 

\begin{enumerate}[(i)]
    \item If $T$ is simplicial and edge stabilisers are finite, then $G$ is virtually cyclic and $T$ is a line. 
    \item If $T$ is an $\R$-tree and (pointwise) arc stabilisers are trivial, then $G$ is virtually abelian and $T$ is a line. 
\end{enumerate}
\end{prop}
\begin{proof}
    If the action is reducible and non-trivial then we can, up to taking an index 2 subgroup of $G$, suppose that there is an invariant end. (See \cite[Corollary 2.3 and Theorem 2.5 (ii)]{Culler1987}; for a dihedral action take the index two subgroup corresponding to the orientation preserving isometries of $\R$.)

    \begin{enumerate}[(i)]
        \item Consider first the case where $T$ is simplicial and edge stabilisers are finite.

    Let $R$ be a ray representing this invariant end. Invariance of the end means that $R \cap Rg$ is an infinite set (in fact, a subray) for every $g \in G$.

    Next observe that for any vertex $p \in R$ there exist only finitely many pairs of vertices $x, y \in R$ such that $p$ is the midpoint of $[x,y]$. This is because both $x$ and $y$ must be equidistant from $p$ and be distinct. But now $ d(x,p) = d(y,p) \leq d(\omega, p)$ where $\omega$ is the intial point of the ray, $R$. There are clearly only finitely many pairs satisfying this on the ray. 

    Armed with this observation, we proceed as follows. Note that for any $g \in G$, $R \cap Rg^{-1}$ is infinite. If we take any point $x \in R \cap Rg^{-1}$ then $x, xg \in R$. But the midpoint of $[x, xg]$ lies in $R \cap A_g$ by Lemma~\ref{treefacts2}(i). Since we have infinitely many possible values for $x$ this implies that $R \cap A_g$ is infinite for any $g \in G$. 

    Therefore if $g$ is elliptic it fixes some subray of $R$ (but potentially a different subray for each $g$) and if $g$ is hyperbolic then the axis of $g$ intersects $R$ in a subray. 

    In any case it is then easy to see that the commutator of any two elements is elliptic with fixed point set containing an infinite subray of $R$. But now any finite set of commutators must fix some common infinite subray of $R$ (one can take the intersection of the fixed subrays for each commutator). Hence any finite set of commutators are contained in some edge stabiliser; for instance the first edge of the common fixed subray. In particular, any finite set of commutators are contained in some edge stabiliser and so generate a finite group as edge stabilisers are finite.

    While it is possible that each finite set of commutators is contained in a different edge stabiliser we will argue that the derived subgroup is finite (and contained in some edge stabiliser) as follows. The action of $G$ on $T$ is co-compact, since $G$ is finitely generated. Hence there exists a constant, $M$, such that any edge stabiliser has order at most $M$. Therefore any finitely generated subgroup of $[G,G]$ has order bounded by $M$ which implies that $[G,G]$ has order bounded by $M$. (As $[G,G]$ is countable and generated by commutators, we can realise it as a countable union of a chain of finitely generated subgroups. Since the order of these subgroups is bounded, the chain of subgroups stabilises and hence $[G,G]$ is finite.)
    
    Therefore $[G,G]$ is finite and hence $G$ is virtually abelian by Lemma~\ref{virtually abelian}. This implies that any two hyperbolic elements of $G$ have positive powers which commute, since sufficiently large positive powers will lie in an abelian subgroup. Therefore, by  Lemma~\ref{treefacts} (v) and Observation~\ref{commute}, all hyperbolic elements have the same axis. This axis is now a $G$-invariant subtree of $T$ by Propostion~\ref{minimal subtree} and so is equal to $T$. Therefore $T$ is a line. 

    To finish we note that some finite index subgroup of $G$ is torsion-free abelian and acts on $T$ with finite (and hence trivial) edge stabilisers. Thus the torsion-free subgroup of $G$ is both free and free abelian (and non-trivial since if $G$ were finite then the action on $T$ would be trivial). Therefore $G$ is virtually infinite cyclic. 
\item The case where $T$ is  an $\R$-tree with trivial arc stabilisers is similar but easier. In this case, once descending to an index two subgroup of $G$ admitting an invariant end, we deduce that any two hyperbolic axes intersect in an infinite ray. Since arc stabilisers are trivial, this means that any two hyperbolic elements commute which implies that their axes are equal. Hence $T$ is isometric to the real line and $G$ is virtually abelian.  

 \end{enumerate}   
\end{proof}

\subsection{Actions and automorphisms}			

\begin{defn}
	Let $G$ be a group, $T$ a $G$-tree and $\varphi \in \aut (G)$. Then $\varphi T$ is the same underlying tree but with the $G$-action twisted by $\varphi$. Specifically, 
	$$
	x {\cdot}_{(\varphi T )} g := x \cdot_T (g \varphi).  
	$$
	Equivalently, if the action on $T$ is given by a homomorphism, $G \xrightarrow{\pi} Isom(T)$, then the action on $\varphi T$ is given by the pre-composition by $\varphi$; \ $G \xrightarrow{\varphi} G \xrightarrow{\pi} Isom(T)$. 
\end{defn}

\begin{rem}
	Since we are writing our action of $\aut(G)$ on $G$ as a right action, this action of $\aut(G)$ on the space of $G$-trees is a left action. Were we to write automorphisms of $G$ on the left, this would be a right action.
\end{rem}

\begin{defn}
	\label{topological representative}
	Let $G$ be a group and $\varphi \in \aut(G)$. A topological representative of $\varphi$ is an equivariant map $f:T \to \varphi T $, from a $G$-tree, $T$, which sends vertices to vertices and edges to edge paths. 
\end{defn}

\begin{obs}
	Note that $f: T \to \varphi T $ being $G$-equivariant means that $f(xg) = f(x) (g \varphi)$ for all $x \in T, g \in G$. 
	
	However, note that if $f$ is a topological representative for $\varphi$, then $f_w$ is a topological representative for $\varphi Ad(w)$, where $f_w(x) = f(x) w$, and $g (\varphi Ad(w)) = w^{-1} (g \varphi) w$. Thus having a topological representaive is really a property of outer automorphisms. 
\end{obs}

\begin{thm}[Culler--Morgan {\cite[Theorem 3.7 (p.586)]{Culler1987}}]
	\label{cullermorgan}
	If $T_1$ and $T_2$ are two minimal, irreducible $G$-trees with $\|g\|_{T_1} = \|g\|_{T_2}$ for all $g \in G$, then there is a unique equivariant isometry, $f:T_1 \to T_2$. 
\end{thm}

\begin{rem}
	Irreducible is a stronger hypothesis than needed but suffices for our purposes and makes the statement a little cleaner. 
\end{rem}

\begin{cor}\label{cor action of Mphi}
	Let $T$ be a minimal irreducible $G$-tree and $\Phi \in \out(G)$. Then $M_{\Phi}$ has an isometric action on $T$ which restricts to the $G$-action on $T$ if and only if $\| g \Phi \|_T = \|g \|_T$ for all $g \in G$. 
\end{cor}
\begin{rem}
	Note that the isomorphism class of $M_{\Phi}$ and the value of $\| g \Phi \|_T$ only depend on the outer automorphism class, so this makes sense. One can think of $g \Phi$ as a conjugacy class rather than a single element. 
\end{rem}

\begin{proof}
	Suppose first that $M_{\Phi}$ acts isometrically on $T$ in a way which extends the $G$-action. Then, since translation length is a conjugacy invariant, 
	$$
	\| g \|_T = \| g^t \|_T = \| g \Phi \|_T \text{ for all } g \in G.
	$$
	
	Conversely, suppose that $\| g \Phi \|_T = \| g \|_T$ for all $ g \in G$. Then, by definition, $\| g\|_{\Phi T } = \|g \Phi \|_T$, which equals $\|g\|_T$ by hypothesis.

	Choose some $\varphi \in \Phi$. Then, by  Theorem~\ref{cullermorgan}, there is an isometry $f:T \to T$ which is equivariant when viewed as an isometry $f: T \to \varphi T$. But now we can extend the $G$ action to an $M_{\Phi}= M_{\varphi}$ action by simply having $t$ act as $f$. The only thing to check is that $t^{-1} g t $ and $g \varphi$ act in the same way on $T$. But this is readily checked to be equivalent to saying that $f: T \to \varphi T$ is equivariant. 
\end{proof}

We can now state the main tools we will use to prove $R_{\infty}$. 

The first we have already stated in Lemma~\ref{lemma Mphi acting on T}. 

\begin{prop}[Levitt--Lustig {\cite[Proposition 3.1]{Levitt2000}}]\label{LL Prop 3.1}
	Let $G$ be a finitely generated group, fix $\Phi\in\out(G)$.
	Let $l$ be the length function of an irreducible action of $G$ on an $\mathbb{R}$-tree.
	If $l\circ\Phi=l$, then $G$ has $R_{\infty}(\Phi)$.
\end{prop}



\begin{prop}[Levitt--Lustig {\cite[Proposition 3.2]{Levitt2000}}]\label{LL Prop 3.2}
	Let $G$ be a finitely generated group, fix $\Phi\in\out(G)$.
	Let $l$ be the length function of an irreducible action of $G$ on an $\mathbb{R}$-tree $T$ and suppose $l\circ\Phi=\lambda l$ where $\lambda>1$.
	If arc stabilisers of $T$ are finite and the number of $\stab(x)$-orbits of directions at branch points $x$ are uniformly bounded, then $G$ has $R_{\infty}(\Phi)$.
\end{prop}


\section{Train Track Maps and free products}

In this section  we show that free products of groups have the $R_{\infty}$ property using the machinery of (Relative) Train Track maps. These were first used in \cite{Bestvina1992} in the context of free group automorphisms and then in \cite{Bestvina2000} to define and analyse the stable (or forward) limit tree. The methods in the former paper were extended to free products in \cite{Collins1994}. However, we are using the treatment in \cite{Francaviglia2011}, \cite{Francaviglia2017} and \cite{Francaviglia2024} because of the more detailed relation to the stable limit tree. We also use some results from \cite{Horbez2017} to deal with the technical conditions mentioned in the introduction; specifically, this allows us to verify the condition that stabilisers of branch point act on the set of directions with finitely many orbits. 

As always, the goal is to achieve condition $R_{\infty}$ by looking at the action of the mapping torus on some tree and applying the argument of \cite[Propositions 3.1 and 3.2]{Levitt2000}.

\subsection{Free factor systems}		

\begin{defn}[Free Factor System]
	Let $G$ be a group which splits as a free product.
	A {free factor system} of $G$ is a pair $\mathcal{F}=(\{G_{1},\dots,G_{k}\},r)$ such that $G=G_{1}\ast\dots\ast G_{k}\ast F_{r}$ where each $G_{i}$ is non-trivial and $F_{r}$ is the free group of rank $r$.
\end{defn}

Following \cite[Definitions 3.2 and 3.12]{Guirardel2007}  we make the following definition.

\begin{defn}
	Let $G$ be a group which splits as a free product and $\mathcal{F}=(\{G_{1},\dots,G_{k}\},r)$ a free factor system for $G$. We denote by $D(\mathcal{F})$ the set of all simplicial $G$-trees (up to equivariant isometry) whose elliptic subgroups are subgroups of conjugates of the $G_i$. That is $T \in   D(\mathcal{F})$ if and only if for every $H \leq G$, $H$ fixes a point in $T$ exactly when $H$ is a subgroup of some conjugate of some $G_i$. 
	
	Further, we let $D_1(\mathcal{F})$ denote the subset of $D(\mathcal{F})$ consisting of those trees whose edge stabilisers are all trivial. 
\end{defn}

\begin{defn}
	Let $\mathcal{F}=(\{G_{1},\dots,G_{k}\},r)$ be a free factor system of some group $G$, and let $\varphi\in\aut(G)$.
	We say that $\mathcal{F}$ is:
	\begin{itemize}
		\item  a {Grushko decomposition} of $G$ if each $G_{i}$ is freely indecomposable and not infinite cyclic.
        \item {proper} if $k+r\ge2$.
		\item {maximal} if for any proper free factor system $\mathcal{F}'=(\{G_{1}',\dots,G_{l}'\},s)$ of $G$, there exists some $i$ so that $G_{i}$ is not contained in any conjugate of any $G_{j}'$.
		
		\item {$\varphi$-invariant} if for each $i$ there is some $j$ and some $g_{j}\in G$ so that $\varphi(G_{i})=G_{j}^{g_{j}}$.
	\end{itemize}
	We say that $\varphi$ is {irreducible} with respect to $\mathcal{F}$ if $\mathcal{F}$ is a maximal, proper, $\varphi$-invariant free factor system of $G$.
\end{defn}

\begin{rem}
	Observe that a Grushko decomposition of a group $G$ is $\varphi$-invariant for any $\varphi\in\aut(G)$. See \cite[III, Proposition 3.7]{Lyndon1977} for a statement of Grushko's Theorem. 
\end{rem}

The goal is to produce a very well behaved topological representative (see Definition \ref{topological representative}) for an automorphism of a free product.

\begin{defn}
	Let $T$ be a simplicial $G$-tree with finitely many (oriented) edge orbits, rerpresented by edges $e_1, \ldots, e_n$. Suppose that $f: T \to \varphi T$ is a topological representative. Then the transition matrix of $f$ is the $n \times n$ matrix, $M_f$, whose $(i,j)$-entry is the number of times the edge path $f(e_i)$ crosses the orbit of $e_j$ (in either direction). 
\end{defn}

\begin{defn}
	Let $T$ be a simplicial $G$-tree with finitely many orbits of edges and $f: T \to \varphi T$ a topological representative for some $\varphi \in \aut(G)$. Then $f$ is called a train track map if for every $k \geq 1$, $f^k$ is locally injective on the interior of edges. 
	
	It is called a metric train track map if, in addition, $T$ admits a $G$-equivariant path metric such that the length of $f(e)$ is $\lambda$ times the length of $e$, where $e$ is any edge of $T$ for some fixed $\lambda\in\mathbb{R}_{\ge0}$. 
\end{defn}

We refer the reader to \cite{Bestvina1992} for the background on how irreducible matrices and their Perron--Frobenius eigenvalues are used in the theory of train tracks. 

From a train track we can construct the (forward) stable limit tree as in \cite{Bestvina2000}, although we will base the following more closely on \cite{Gaboriau1998}. Namely, given a metric train track map, $f: T \to \varphi T$ with associated constant $\lambda > 1$, set $d_{\infty} (x,y) = \lim_{n \to \infty} \frac{d_T(f^n(x), f^n(y))}{\lambda^n}$, for $x, y \in T$. This converges due to monotonicity and is a pseudo-metric on $T$. Therefore on taking a suitable quotient, one obtains a metric space $T_{\infty}$ with metric (which we will still name), $d_{\infty}$. 

It is straightforward to see that the quotient map, $T \to T_{\infty}$ is $G$-equivariant and 1-Lipschitz and that $T_{\infty}$ is a $0$-hyperbolic metric space on which $G$ acts isometrically. In other words, $T_{\infty}$ is an $\R$-tree on which $G$ acts. Moreover, the train track property immediately implies that this quotient map is an isometry restricted to the edges of $T$. 

It also follows quickly that for any $\gamma \in G$, $\|\gamma\|_{T_{\infty}} = \lim_{n \to \infty} \frac{\|\gamma \varphi^n\|_T}{\lambda^n}$. It is well known that given a (metric) train track there is a hyperbolic $g \in G$ such that $f^k$ is injective on $A_g$ for all $0 <k \in \Z$, \cite[Corollary 8.12]{Francaviglia2011} (this is also found in \cite{Bestvina1992}). Such a $g$ is called a `legal' group element. It now follows that if $g$ is legal, $0 < \| g \|_{T_{\infty}} = \|g\|_T$. 

One can now proceed as in \cite[Lemma II.7]{Gaboriau1998}, to see that $T_{\infty}$ is a minimal and non-trivial $G$-tree. In fact we record more.

\begin{thm}
	\label{train tracks}
	Let $G$ be a non-trivial free product which is not the infinite dihedral group. Let $\Phi \in \out(G)$.


Then there exists a maximal, proper, $\varphi$-invariant free factor system $\mathcal{F}$ and a simplicial $G$-tree $T \in D_1(\mathcal{F})$ admitting a train track representative $f$ for $\varphi$, with associated constant $\lambda \geq 1$. Furthermore:
    
	\begin{enumerate}[(i)]
		\item $\lambda$ is the Perron--Frobenius eigenvalue of the transition matrix for $f$. 
		\item $\lambda = 1$ if and only if $f$ is an isometry. In this case the mapping torus $M_{\varphi}$ acts isometrically on $T$.  
		\item If $\lambda >1$ then $\lim_{n \to \infty} \frac{\varphi^n T} {\lambda^n}$ is a $G$-tree, $T_{\infty}$, called the stable limit tree for $\varphi$. 
		
		\item For all $g \in G$, $\|g \varphi\|_{T_{\infty}} = \lambda \| g \|_{T_{\infty}}$. In particular, there exists a $G$-equivariant homothety, $H: T_{\infty} \to \varphi T_{\infty}$ of stretching factor $\lambda$, by Theorem~\ref{cullermorgan}. 
        \item The stable limit tree is a non-trivial, irreducible $\R$-tree. 
		\item If $1 \neq g \in G$ acts elliptically on $T$, then $g$ fixes a unique point of $T_{\infty}$. In particular, non-trivial elements of arc stabilisers of $T_{\infty}$ act hyperbolically on $T$. 
        \item There exist finitely many segments, $f_1, \ldots, f_r$ in $T_{\infty}$ whose union of $G$-orbits covers the whole of $T_{\infty}$. 
	\end{enumerate}
\end{thm}

\begin{proof}

As in \cite[Corollary 8.25]{Francaviglia2011}, there exists a maximal  $\varphi$-invariant free factor system, $\mathcal{F}$ so that $\varphi$ is $\mathcal{F}$-irreducible. Again using \cite[Theorem 8.24]{Francaviglia2011}, $\varphi$ will admit a train track representative. Just as in \cite{Bestvina1992}, the transition matrix of this map has a Perron--Frobenius eigenvalue, $\lambda \geq 1$, which is equal to the Lipschitz constant of the train track map (in \cite{Bestvina1992} one must choose lengths for the edges, whereas in \cite{Francaviglia2011} this is implicit in the definition. This is the distinction between topological and metric train track maps). 

Moreover, the theory of irreducible matrices implies that $\lambda=1$ if and only if the transition matrix is a permutation matrix, implying that $f$ is an isometry. 

This proves (i) and (ii). For (iii) and (iv) we refer the reader to \cite[Lemma 2.14.1]{Francaviglia2024}. Point (v) is well known to experts but we can also deduce it from Lemma~\ref{trivial arc} and Theorem~\ref{irreducible finite edge stabs}, since the only non-trivial virtually abelian free product is the infinite dihedral group. Point (vi) is \cite[Lemma 2.13.6]{Francaviglia2024}.

Finally, for part (vii), we argue as above (following \cite{Gaboriau1998}, but see also \cite[Section 2.13]{Francaviglia2024}). That is, since the quotient map $T \to T_{\infty}$ is an isometry restricted to edges of $T$ we may take the $f_i$ to be the isometric images of the edges of $T$. 

\end{proof}

The following are also well-known to experts and the proofs are the same as in \cite{Gaboriau1998}, adapted to the free product case.

\begin{lemma}
\label{arc segments}
For any $p \in \N$ there exists a constant $L_p$ such that any arc $\sigma$ in $T_{\infty}$ whose length is  greater than $L_p$, will contain at least $p$ non-trivial, disjoint subsegments which are all in the same $G$-orbit. 
\end{lemma}
\begin{proof}
     By Theorem~\ref{train tracks} (vii), there are finitely many arcs, $f_1, \ldots, f_r$ whose $G$-orbits cover $T_{\infty}$. Moreover, any arc segment $\sigma$ can be subdivided into subsegments, $\sigma = z_1 \cup \ldots \cup z_n$ such that for each $i$ there is a group element $g_i$ such that $z_i g_i$ is a subsegment of one of $f_1, \ldots, f_r$. As this is a subdivision, the length of $\sigma$ is equal to the sum of the lengths of the $z_i$ and two distinct $z_i$ can intersect in at most a single point. (If numbered sequentially, $z_i$ intersects precisely $z_{i-1}$ and $z_{i+1}$ in a single point and is disjoint from the rest).  
	
	Now, after possibly further subdividing, we may assume that for $i \neq j$ either $z_ig_i = z_j g_j$ or their intersection is at most a single point. 
	
	Next let us suppose that there is an integer $p$ such that, 
	$$
	\max_i |\{ j \ : z_ig_i = z_j g_j \ \}| \leq p.
	$$
	It is then clear that the sum of the lengths of the $z_i g_i$ is at most $p(l_1+\ldots+l_r)$, where $l_j$ is the length of $f_j$. But since each $g_i$ is an isometry, this means that the length of $\sigma$ is also at most $p(l_1+\ldots+l_r)$. 
	
	Therefore, if $\sigma$ is any segment of length greater than $3p(l_1+\ldots+l_r)$ then we can ensure that $\sigma$ contains at least $3p$ subsegments, $c_1, \ldots ,c_{3p}$ which are all in the same $G$ orbit. Since each $z_i$ meets at most $z_{i-1}$ and $z_{i+1}$, we deduce that at least $p$ of these must be disjoint. 
\end{proof}

\begin{lemma}
	\label{trivial arc}
	If $G$ is finitely generated, then arc stabilisers in $T_{\infty}$ are trivial.
\end{lemma}
\begin{rem}
    The hypothesis that $G$ be finitely generated is not needed here. It is used to deduce that the limit of small actions is small. However, this will always be true of a limit of edge-free actions and we include the hypothesis for convenience of referencing. 
\end{rem}

\begin{rem}
    Note that we are implicitly assuming that the $\lambda$ from Theorem~\ref{train tracks} is greater than 1. 
\end{rem}

\begin{proof} Consider some non-trivial arc $c$ of $T_{\infty}$ and let $S$ be the pointwise stabiliser of $c$. By Theorem~\ref{train tracks}(vi), every non-trivial element of $S$ acts on $T$ as a hyperbolic isometry. In particular, $S$ acts freely on the simplicial tree, $T$, and so is a free group. Thus, in order to show that $S$ is trivial it is enough to show that $S$ is finite. 

We therefore argue by contradiction and suppose that $S$ is infinite.

	Since $T\in\mathcal{D}_{1}(\mathcal{F})$ then edge stabilisers in $T$ are trivial, hence the action of $G$ on $T$ (and also on each $\varphi^{n}T$) is {small}.
	Since a limit of small actions is itself small, by Culler--Morgan \cite[Theorem 5.3]{Culler1987}, $S$ cannot contain a free subgroup of rank 2.

    However, we have already noted that $S$ is a free group and so for $S$ to be infinite it would have to be an infinite cyclic group. Therefore the minimal $S$-invariant subtree of $T$ is a line $L$ on which $S$ acts by translation. (Note that all the non-trivial elements of $S$ have axis equal to $L$, by Lemma~\ref{treefacts} (v)). 


	
	Let $\bar{S}$ be the (setwise) stabiliser of $L$. Then $\bar{S}$ is virtually cyclic as its minimal invariant tree is $L$, by Lemma~\ref{treefacts} (iv). Moreover, it is also a maximal virtually cyclic group as any virtually cyclic group containing $S$ must preserve $L$, since any hyperbolic element in such a subgroup has a power which lies in $S$ and hence has axis equal to $L$. (Recall that an invariant subtree for a subgroup is the union of all the hyperbolic axes, by Proposition~\ref{minimal subtree}.)

    Let $H:T_{\infty}\to \varphi T_{\infty}$ be the homothety described in Theorem \ref{train tracks}(iv) and consider $H^{k}(c)$ for some $k$. Let $p$ be the index of $S$ in $\bar{S}$. 
    
    By Lemma~\ref{arc segments}, we may choose $k$ sufficiently large so that $H^k(c)$ contains $p+1$ disjoint non-trivial arcs which are in the same $G$-orbit. Since $H$ is equivariant, $c$ will also contain $p+1$ disjoint arcs, $c_0, \ldots, c_p$ in the same $G$ orbit. Hence there exist group elements, $g_1, \ldots, g_p$ such that $c_0 g_i = c_i$. (Set $g_0=1$.)

Let $C_i$ denote the pointwise stabiliser of $c_i$. As argued above, each $C_i$ is cyclic (since it is small and acts freely on $T$). Note that $S$ is a subgroup of $C_i$ and hence $C_i$ is a subgroup of $\bar{S}$, since the latter is a maximal virtually cyclic subgroup of $G$. 

Moreover, since $c_0 g_i = c_i$ we also have that $C_i = C_0^{g_i}$. But now $S \leq C_0^{g_i} \leq \bar{S}^{g_i}$. By maximality, $\bar{S} = \bar{S}^{g_i}$ for all $i$. However, any hyperbolic element of $\bar{S}$ has axis equal to $L$, whereas any hyperbolic element of $\bar{S}^{g_i}$ has axis equal to $L g_i$. Hence $L=Lg_i$ and $g_i \in \bar{S}$ for all $i$. But since the $c_i$ are disjoint, $g_0, \ldots, g_p$ are in distinct cosets of $C_0$ in $\bar{S}$, contradicting the fact that $S \leq C_0$ has index $p$ in $\bar{S}$. This contradiction shows that $S$ must in fact be trivial.

\end{proof}

Using this, we can bound the number of orbits of directions using the following result. 

\begin{prop}[Horbez {\cite[Proposition 4.4]{Horbez2017}}] \label{Horbez index}




    Let $G$ be a countable group, $\mathcal{F}$ a free factor system for $G$, and $T$ an $\R$-tree equipped with an isometric $G$ action such that any subgroup in $\mathcal{F}$ fixes a point in $T$. (The subgroups in $\mathcal{F}$ are elliptic in $T$). 

    Then if pointwise arc stabilisers in $T$ are trivial, there is a uniform bound on the number of $\stab(x)$-orbits of directions at branch points $x$ in $T$.
\end{prop}

\begin{rem}
	Horbez \cite[Proposition 4.4]{Horbez2017} gives a bound on the `index' of a tree $T$, a component of which is the sum over $G$-orbits of points $x$ in $T$ of the number of $\stab(x)$-orbits of directions from $x$ in $T$ with trivial stabiliser. The proposition then follows immediately.
\end{rem}

Combining this with Proposition~\ref{LL Prop 3.2} we get,


\begin{thm}
    \label{thm free product has R infinity}
	Let $G$ be a finitely generated group which splits as a non-trivial free product of finitely many indecomposable groups. Then $G$ has property $R_{\infty}$.
\end{thm}

\begin{rem}
    The only times we need the hypothesis of being finitely generated are when we invoke the \cite[Theorem 5.3]{Culler1987} in Lemma~\ref{trivial arc} and \cite[Proposition 4.4]{Horbez2017} in Proposition~\ref{Horbez index}. However, both results are true in general without this hypothesis. 
\end{rem}
\begin{proof}
	The case where $G=D_{\infty}$ is known and can be dealt with separately (see Theorem~\ref{D infty}). 
	
	In all other cases, every tree in $D_1(\mathcal{F})$, for any proper free factor system $\mathcal{F}$ is irreducible. 
	
	Consider some $\varphi \in \aut(G)$. We now use Theorem~\ref{train tracks}.  If $\lambda =1$, we get an isometric action of $M_{\varphi}$ on some $T \in D_1(\mathcal{F})$ and hence we are done by Proposition~\ref{LL Prop 3.1} (or Lemma~\ref{lemma Mphi acting on T}). If $\lambda > 1$ we form the stable limit tree, $T_{\infty}$. 
	
	All that is left is to verify the hypotheses of Proposition~\ref{LL Prop 3.2} in the case where $\lambda > 1$. This is done in Theorem~\ref{train tracks}, Lemma~\ref{trivial arc} and Proposition~\ref{Horbez index}. 
\end{proof}

\undecidable

\begin{proof}
    It is well known \cite[Chapter IV, Theorem 4.1]{Lyndon1977} that the property of being trivial is undecidable amongst finitely presented groups. (It is clear that triviality is a Markov property). 

Suppose we are given a finite presentation of a group, $G = \langle X \ : \ R  \rangle $. 

Then we can form the finitely presented group, $\Gamma = \Z * G = \langle X, t \ : \ R \rangle$, where $t$ is a letter which does not appear in $X$. 

The group $\Gamma$ is isomorphic to $\Z$ precisely when $G$ is trivial. In which case $\Gamma$ would not have $R_{\infty}$. ($\Z$ has exactly two twisted conjugacy classes with respect to the automorphism  sending $n \mapsto -n$). 

On the other hand, if $G$ were non-trivial, then $\Gamma$ would be a finitely generated group which splits as a non-trivial free product of finitely many indecomposable groups and so has $R_{\infty}$ by Theorem~\ref{thm free product has R infinity}. 

Therefore, if we could decide $R_{\infty}$ for $\Gamma$ we could also decide triviality for $G$. 
    
\end{proof}


\section{ \texorpdfstring{$R_{\infty}$} \ \  for Groups with Infinitely Many Ends}


\subsection{Accessible groups}

In this section we extend the result to groups with infinitely many ends. We recall the definition of an end. (See also \cite[IV, Definition 6.4]{Dicks2010}, for an alternate but equivalent formulation.)

\begin{defn}
    Let $G$ be a finitely generated group with finite generating set $S$. Let $\Gamma$ be the Cayley graph of $G$ with respect to $S$. Then the number of ends of $G$ is the supremum of the number of infinite components of the graphs, $\Gamma \setminus F$, where $F$ ranges over all finite subgraphs of $\Gamma$. 
\end{defn}

It turns out that the number of ends does not depend on $S$ and is an invariant of the group. In fact, more can be said via the classical Stallings Theorem on ends, which can be found in \cite[IV, Theorem 6.10 and Theorem 6.12]{Dicks2010}

\begin{thm}
\label{stallings}
    Let $G$ be a finitely generated group.
    \begin{itemize}
        \item The number of ends of $G$ is $0,1,2$ or $\infty$. 
        \item If $G$ has more than one end, then $G$ acts non-trivially on a simplicial tree with finite edge stabilisers. 
        \item The number of ends of $G$ is 2 if and only if $G$ has a infinite cyclic subgroup of finite index. 
    \end{itemize}
\end{thm}

Thus a group with infinitely many ends can be `split' along finite subgroups. One could then ask whether the vertex stabilisers of such an action split further and whether this process terminates. The groups for which it does terminate are called \textit{accessible}.

\begin{defn}[See {\cite[p.189]{Scott}}  and {\cite[IV, Definition 7.1]{Dicks2010}}] 
\label{accessible}
Given a group $G$ a simplicial $G$-tree $T$ is called terminal if the vertex stabilisers of $T$ have at most one end and all edge stabilisers are finite. 

A finitely generated group $G$ is called accessible if it admits a terminal $G$-tree. 
\end{defn}
\begin{obs}
\label{finite generation}
    We insist that an accessible group is finitely generated as in \cite{Scott} but contrary to \cite{Dicks2010}. 

Note that the vertex groups of an accessible group $G$ acting minimally on a terminal $G$-tree $T$ are finitely generated, by \cite[III, Lemma 8.1]{Dicks2010}.

\end{obs}

\begin{defn}[{\cite[VI, Definition 4.1]{Dicks2010}}]
	A group $G$ is called {almost finitely presented} if it is of type $FP_2$ over $\Z_2$. In particular, every finitely presented group is almost finitely presented. 
\end{defn}


\begin{rem}
    Recall that a reduced $G$-tree (Definition~\ref{definition of reduced}) is also minimal by  Lemma~\ref{reduced implies minimal}. 
\end{rem}

\begin{thm}[\cite{Dicks2010}, VI, Theorem 6.3]
	Let $G$ be an almost finitely presented group. Then $G$ has a reduced terminal $G$-tree $T$. Moreover, the action on $T$ is co-compact. 
\end{thm}

\begin{proof}
    The existence of a terminal $G$-tree follows from \cite[VI, Theorem 6.3]{Dicks2010}. The fact that we can take the tree to be reduced follows from \cite{Guirardel2007}. The action can then be taken to be minimal by Lemma~\ref{Serre lemma} and Proposition~\ref{minimal subtree} (as $G$ is finitely generated). Finally, the action is co-compact, as is any action of a finitely generated group acting minimally on a simplicial tree. For instance, see \cite[I, Proposition 4.13]{Dicks2010}. 
 \end{proof}

We will also need the following, which is just a restatement of the `blowing up' construction given in \cite[IV, Section 7]{Dicks2010}.

\begin{prop}
\label{blow up}
    Let $H$ be an accessible group acting on a simplicial tree $S$ with finite edge stabilisers. Then there exists a terminal $H$-tree, $X$ so that every edge stabiliser of $S$ is equal to some edge stabiliser of $X$. 

    Moreover, if the action on $S$ is non-trivial then so too is the action on $X$. 
\end{prop}
\begin{proof}
    The idea here is to `blow up' $S$ to produce $X$. Just as in Observation~\ref{finite generation}, all the vertex stabilisers in $S$ are finitely generated and hence accessible, since $H$ is. If we write $H$ as a graph of groups using $S$, we can then replace each vertex group, $H_v$, with a graph of groups whose edge stabilisers are finite and whose vertex stabilisers have at most one end. In particular, since all edge stabilisers in $S$ are finite, each edge group must fix a vertex in any incident $H_v$ terminal tree and so we get a well defined graph of groups giving us the result. 
\end{proof}

\begin{prop}
\label{order of edge stabs}
    Suppose that a finitely generated group $H$ acts minimally on a simplicial tree $T$ with finite edge stabilisers and consider the reduction, $r(T)$ obtained from $T$ by collapsing all collapsible edges. Then, 

    \begin{enumerate}[(i)]
        \item If $T$ is not a point, then $r(T)$ is also not a point. 
        \item For any edge $e$ of $r(T)$ there exists an edge $e'$ of $T$ such that $|\stab_H(e')| =|\stab_H(e)|$.


    \end{enumerate}
\begin{proof}
    Since $T$ and $r(T)$ are in the same deformation space by Observation~\ref{reduced tree facts}, if $r(T)$ is a point then the whole of $H$ is elliptic in $r(T)$ and hence also in $T$, proving (i). 

For (ii), we can simply regard the edge set of $r(T)$ as a subset of the edge set for $T$, from which the result immediately follows.

 \end{proof}
    
\end{prop}

\begin{defn}
	A subgroup $H$ of a group $G$ is called {characteristic} if for every $\varphi\in\aut(G)$ we have that $(H)\varphi\le H$. 
\end{defn}

\begin{obs}\label{obs characteristic properties}
	It is an easy exercise to see that if $H$ is a characteristic subgroup of a group $G$, then the following all hold:
	\begin{enumerate}
		\item For any $\varphi\in\aut(G)$ we have $(H)\varphi=H$.
		\item $H$ is a normal subgroup of $G$.
		\item Every automorphism of $G$ induces an automorphism of $G/H$.
		\item If $G$ has a unique subgroup $K$ of a given order, then $K$ is characteristic.
	\end{enumerate}
\end{obs}

\begin{lemma}\label{lemma G/N R-infiity implies G R-infinity}
	Suppose a group $G$ has a characteristic subgroup $N$.
	If $G/N$ has $R_{\infty}$, then so too does $G$.
\end{lemma}

\begin{proof}
	Let $\varphi\in\aut(G)$. Since $N$ is characteristic, $\varphi$ induces an automorphism $\varphi_{\ast}$ on $G/N$ given by $(g\cdot N)\varphi_{\ast}=(g)\varphi\cdot N$.
	Given $xN, yN\in G/N$, we have that
	\begin{align*}
		&	xN\sim_{\varphi_{\ast}} yN	\\
		\Longleftrightarrow	&	\exists wN\in G/N \text{ such that } (wN)\varphi_{\ast}xN(wN)^{-1}=yN		\\
		\Longleftrightarrow	&	\exists w\in G \text{ such that } (w\varphi)xw^{-1}N=yN		\\
		\Longleftrightarrow	&	\exists w\in G \text{ and } \exists n\in N \text{ such that } (w\varphi)xw^{-1}=ny	\\
		\Longleftrightarrow	&	\exists n\in N \text{ such that } x\sim_{\varphi} ny	.
	\end{align*}
	In particular, if $xN$ and $yN$ belong to different $\sim_{\varphi_{\ast}}$ classes in $G/N$, then $x$ and $y$ belong to different $\sim_{\varphi}$ classes in $G$ (since $1\in N$).
	So if $G/N$ has $R_{\infty}(\varphi_{\ast})$, then we must have that $G$ has $R_{\infty}(\varphi)$.
	This holds for all $\varphi\in\aut(G)$, hence if $G/N$ has $R_{\infty}$ then so too must $G$.
\end{proof}

\begin{lemma}\label{lemma kernel is characteristic}
	Let $G$ be a group acting on a terminal minimal tree $T$, and let $N=\{ g\in G | xg=x \ \forall x\in T \}$ be the kernel of this action. Then $N$ is the maximal (normal and finite) subgroup of $G$. 
	In particular, $N$ is a finite characteristic subgroup of $G$.
\end{lemma}

\begin{proof}
	First, observe that $N$ must be a subgroup of every edge stabiliser of $T$.
	Since $T$ is terminal, it has finite edge stabilisers, hence $N$ must be a finite group.
	Since $N$ is a kernel, then it must be normal.
	
	We will now show that $N$ is the maximal normal finite subgroup of $G$.
	Indeed, let $K$ be another finite normal subgroup in $G$.
	Since $K$ is finite, then it must fix some point $p$ in the tree $T$.
	Let $g\in G$ and $k\in K$. Since $K$ is normal then there exists $k'\in K$ with $g^{-1}k'g=k$.
	Then $(p\cdot g)\cdot k=p\cdot gk=p\cdot k' g=p\cdot G$.
	Thus $K$ pointwise fixes $p\cdot G$, and so must also fix the convex hull of the points $p\cdot G$.
	Since $T$ is assumed to be minimal, we must have that this convex hull is equal to $T$, and hence $K\le N$.
	Hence $N$ is the maximal normal finite subgroup of $G$.
	
	Now by Observation \ref{obs characteristic properties} (4), we must have that $N$ is characteristic in $G$.
\end{proof}

\begin{cor}\label{cor G/N R-infiity implies G R-infinity}
	Let $G$ be a group acting on a terminal minimal $G$-tree $T$, and let $N=\{ g\in G | xg=x \ \forall x\in T \}$ be the kernel of this action.
	If $G/N$ has $R_{\infty}$, then so too does $G$.
\end{cor}

\begin{proof}
	This follows immediately from Lemmas \ref{lemma G/N R-infiity implies G R-infinity} and \ref{lemma kernel is characteristic}.
\end{proof}

\begin{rem}
	Since $N$ is finite and $G$ is  {assumed to be finitely generated}, then $G$ is quasi-isometric to $G/N$.
	Thus $G$ has infinitely many ends if and only if $G/N$ has infinitely many ends.
\end{rem}

\begin{defn}
\label{Tq tree}
	Let $G$ be an accessible group and $T$ a reduced terminal $G$-tree. 
	\begin{enumerate}[(i)]
		\item Let $q:=q(G)$ be the order of the smallest edge stabiliser of $T$. By \cite[Corollary 7.3]{Guirardel2007} this does not depend on $T$.
		\item Let $T(q)$ be the tree obtained from $T$ by collapsing every edge whose stabiliser has order greater than $q$.
        \item Let $V_{0}$ be the set of vertices $x\in T(q)$ for which there exist incident edges $e_{1}$ and $e_{2}$ with $\stab(e_{1})\ne\stab(e_{2})$ as subgroups of $G$.
        Let $V_{1}$ be the set of subgroups $Y$ of $G$ which occur as an edge stabiliser in $T(q)$.
        We construct a bipartite tree $T(q)_{c}$ whose vertex set is $V(T(q)_{c})=V_{0}\sqcup V_{1}$, and where there is an edge $(x,Y)$ between some $x\in V_{0}$ and some $Y\in V_{1}$ if and only if there is some edge $e\in T(q)$ with endpoint $x$ and stabiliser $Y$.        
	\end{enumerate}
\end{defn}

\begin{rem}
    Note that $T(q)_{c}$ is the tree of cylinders constructed by Guirardel and Levitt \cite[Section 4.1]{Guirardel2011}, where the admissible relation is equality of edge stabilisers (as in \cite[Example 3.6]{Guirardel2011}).
\end{rem}

We refer the reader to \cite{Guirardel2011} for the construction of the tree of cylinders. However we note the following straightforward exercise.

\begin{prop}
\label{cylinder facts}
    Let $G$ be a group acting on a simplicial tree $T$, subject to some admissable relation $\sim$ with corresponding tree of cylinders, $T_c$. 
    \begin{enumerate}[(i)]
        \item If the action on $T$ is minimal, then so is the action on $T_c$
        \item If the action on $T$ is minimal and non-trivial, then the action on $T_c$ is trivial exactly when all edges belong to a single cylinder. Equivalently, this is exactly when all edge groups are equivalent under $\sim$. 
        \item If the action on $T_c$ is non-trivial and reducible, then so is the action on $T$. 
    \end{enumerate}
\end{prop}

\begin{proof}
    \begin{enumerate}[(i)] 
\item    Each vertex of $T_c$ is either a vertex of $T$ belonging to more than one cylinder (the $V_0$ vertices) or a cylinder in $T$ (the $V_1$ vertices). It is straightforward to show that the action of $G$ on $T$ then induces an action of $G$ on $T_c$. For minimality see  \cite[Section 4.1, p.13]{Guirardel2011} or \cite[Lemma 4.9]{Guirardel2004}. 
        \item If the action of $G$ on $T_{c}$ is minimal, then the action on $T_{c}$ is trivial if and only if $T_{c}$ is a single point, i.e. the vertex set of $T_{c}$ consists of a single cylinder $Y\in V_{1}$.
        \item We briefly argue that if $T_{c}$ admits a fixed end or an invariant line, then so too must $T$. Suppose that $T_c$ admits an invariant line $L$. Since $T_c$ is bi-partite, we can look at the vertices of $T_c$ which are vertices of $T$ belonging to more than one cylinder. (That is, the vertices in $V_0$). Thus we can think of $L \cap V_0$ as a set of vertices of $T$. It is clear that the convex hull of these is a line, $\Tilde{L}$ in $T$. The invariance of $L$ implies the invariance of $\Tilde{L}$. 

        The case of an end is similar.
    \end{enumerate}
\end{proof}

\begin{prop}
\label{stabilisers}
    Let $G$ be an accessible group, $T$ a reduced terminal $G$-tree and $T(q)$ the $G$-tree obtained from $T$ as in Definition~\ref{Tq tree}. Then, 
    \begin{enumerate}[(i)]
        \item Every edge stabiliser of $T(q)$ is finite of order $q$. 
        \item $T(q)$ is a reduced $G$-tree. 
        \item Every vertex stabiliser of $T(q)$ is accessible. 
        \item Suppose that $H$ is a vertex stabiliser of $T(q)$ and $\varphi \in \aut(G)$. Then $H \varphi$ also fixes a vertex of $T(q)$
    \end{enumerate}
\end{prop}
\begin{proof}
    Let $F$ denote the subforest of $T$ consisting of all edges whose full stabiliser has order greater than $q$ and their incident vertices. We can think of the edges of $T(q)$ as simply being the edges of $T$ which are not in $F$, proving (i).

 Further, note that $F$ is $G$-invariant. Hence if $C$ is a component of $F$ and $g \in G$, then $C \cap Cg \neq \emptyset$ implies that $C = Cg$, since $Cg$ must be another component of $F$. The vertex stabilisers of $T(q)$ are then either vertex stabilisers of $T$ or stabilisers of components of $F$. It is then easy to see that $T(q)$ is reduced and so (ii) holds. 

Next we want to argue that vertex stabilisers of $T(q)$ are finitely generated. This is clear if the vertex of $T(q)$ is equal to a vertex in $T$. 

Otherwise, consider a component $C$ of $F$ and let $H$ be its setwise stabiliser - this is a vertex stabiliser in $T(q)$. As above, $C \cap Cg \neq \emptyset$ implies that $C = Cg$. In particular, $G$-orbits on $C$ must equal $H$-orbits and $G$-stabilisers in $C$ must equal $H$ stabilisers. Therefore the action of $H$ on $C$ is reduced and hence minimal, so $C$ is the minimal $H$-invariant subtree $T_H$. Moreover, since $G$ acts with finitely many orbits, so does $H$ on $T_H$, and since $G$-stabilisers are finitely generated, so are $H$ stabilisers in $C$. Hence $H$ is finitely generated by the fundamental theorem of Bass--Serre Theory (see e.g. \cite[I, Theorem 4.1]{Dicks2010}).



We also argue that each vertex stabiliser of $T(q)$ is accessible. If the vertex stabiliser of $T(q)$ equals the vertex stabiliser of some vertex in $T$, then it has at most one end and the result is clear. Otherwise, a vertex stabiliser $H$ will be the setwise stabiliser of some component $C$ of $F$. But the action of $H$ on $C$ is co-compact, has finite edge stabilisers and the vertex stabilisers have at most one end. This proves (iii). 

Finally we prove (iv). First consider the action of $H$ on $T$. If $H$ is a vertex stabiliser of $T$, then it has at most one end and is finitely generated. Hence $H \varphi$ is also finitely generated and has at most one end. Therefore, since $T(q)$ has finite edge stabilisers, $H \varphi$ must fix a vertex of $T(q)$. 

If $H$ does not fix a vertex of $T$, then it is the setwise stabiliser of some component, $C$ of $F$ and, as argued above, is finitely generated and $C=T_H$, the minimal invariant subtree of $H$ which is reduced and whose edge stabilisers have order greater than $q$. 


Now, consider the action of $H\varphi$ on $T(q)$. Suppose that $H \varphi$ does not fix a vertex of $T(q)$, then as $H\varphi$ is finitely generated it admits an invariant subtree, $S$. By Proposition~\ref{blow up}, $H \varphi$ admits a non-trivial terminal tree $X$ whose edge stabilisers have order at most $q$ (since $G$-edge stabilisers in $T(q)$ have order $q$). 

Next we reduce $X$ to obtain $r(X)$. By Proposition~\ref{order of edge stabs}, the action of $H \varphi$ on $r(X)$ is reduced and non-trivial and edge stabilisers have order at most $q$. But since $H$ is isomorphic to $H \varphi$, this gives us two reduced terminal $H$-trees with different edge stabilisers, contradicting Observation~\ref{reduced tree facts}. This proves (iv).

\end{proof}

\begin{thm}
	\label{infinitely many ends}
	Let $G$ be an accessible group with infinitely many ends and let $N$ be the maximal normal finite subgroup of $G$. Then: 
	\begin{enumerate}[(i)]
		\item For any reduced terminal $G$-tree $T$ the deformation space of $T(q)$ is $\aut(G)$ invariant.
		\item If $G/N$ does not split as a non-trivial free product then the tree of cylinders $T(q)_c$ is a non-trivial, irreducible, minimal $G$-tree which is invariant under the action of $\aut(G)$. (Here the admissable relation is equality.)
	\end{enumerate}
\end{thm}

\begin{proof}

Let $\mathcal{D}$ be a deformation space of $G$-trees and $\varphi \in \aut(G)$. Then $\mathcal{D} \varphi$ consists of all simiplicial $G$-trees for which $H \varphi$ is elliptic for any $\mathcal{D}$-elliptic subgroup $H$. 

Therefore, to prove (i), it is enough to show that for vertex stabiliser $H$ of $T_q$ and any automorphism $\varphi$ of $G$ that $H \varphi$ also stabilises a vertex of $T_q$. This is Proposition \ref{stabilisers}(iv). This proves (i). 

    \medskip

    To prove (ii), we first consider the general situation.

  Note that by \cite{Guirardel2011}, equality is an admissable relation since all edge stabilisers in $T(q)$ have the same order. By Proposition~\ref{stabilisers}, we know that $T(q)$ is reduced and hence minimal by Lemma~\ref{reduced implies minimal}. Therefore, $T(q)_c$ is minimal. If $T(q)_c$ were trivial, this would imply that all edge stabilisers are equivalent, which in our context means equal (by Proposition~\ref{cylinder facts}). In particular, any edge stabiliser would be normal and so $G/N$ would act on the tree $T_q$ with trivial edge stabilisers. 

Hence, if $G/N$ does not split as a free product, then the action of $G$ on $T(q)_c$ is non-trivial. Furthermore, we know that the action of $G$ on $T(q)$ is irreducible by Proposition~\ref{irreducible finite edge stabs} and hence the action on $T(q)_c$ is also irreducible by Proposition~\ref{cylinder facts}. The fact that $T(q)_c$ is invariant under the action of $\aut(G)$ follows from the fact that the deformation space of $T(q)$ is $\aut(G)$ invariant and \cite[Proposition 4.11]{Guirardel2011} .
\end{proof}



\mainends

\begin{proof}
	Let $N$ be the maximal normal finite subgroup of $G$. If $G/N$ is a non-trivial free product then $G$ has $R_{\infty}$ by Corollary~\ref{cor G/N R-infiity implies G R-infinity} and Theorem \ref{thm free product has R infinity}. 
	
	Otherwise, $T(q)_c$ is a non-trivial, irreducible, minimal $G$-tree which is invariant under the action of $\aut(G)$ by Theorem~\ref{infinitely many ends} and hence $G$ has $R_{\infty}$ by Proposition~\ref{LL Prop 3.1}. 
\end{proof}

\section{Relatively hyperbolic groups}

\subsection{JSJ decompositions and invariant trees}

\begin{thm}
\label{rel hyp and peripheral}
    Let $G$ be a non-elementary finitely presented hyperbolic group relative to a family $\mathcal{P}= \{ P_1, \ldots, P_n\} $ of finitely generated groups. Let $\out(G; \mathcal{P})$ denote the subgroup of $\out(G)$ preserving the conjugacy classes of the $P_i$, but allowing permutation of the $P_i$. 

    Then $G$ has infinitely many twisted conjugacy classes with respect to any $\Phi \in \out(G; \mathcal{P})$. 
\end{thm}
\begin{rem}
    Note that our definition of $\out(G; \mathcal{P})$ differs from that of \cite{Guirardel2015} in that we are allowing permutations. However, this does not effect the invariance of the JSJ tree and so this is a benign change to make. We also note that the hypothesis that the $P_i$ are finitely generated arises since we use the JSJ decomposition of \cite{Guirardel2016}, where that is required. 
\end{rem}

\begin{proof}
In our context $G$ being non-elementary means that it is not virtually cyclic and it is not equal to a $P_i$.

If $G$ has infinitely many ends, we can invoke Theorem~\ref{inf ends has rinf} to deduce that $G$ has $R_{\infty}$. If $\Phi$ has finite order, we can use the argument of Delzant as in \cite[Proposition 3.3 and Lemma 3.4]{Levitt2000}. We note that the argument there is for hyperbolic, rather than relatively hyperbolic groups, but the same arguments work since relative hyperbolicity is a commensurability (in fact a quasi-isometry) invariant by \cite[Theorem 1.2]{Drutu2009}, and one can replace the arguments about infinite order elements with ones about loxodromic elements in the relative hyperbolic setting. 

Thus the remaining case is where $\out(G; \mathcal{P})$ has an element of infinite order and $G$ is one-ended. By \cite[Corollary 9.20]{Guirardel2016}, there is a canonical relative JSJ tree, $T$. The fact that it is relatively canonical means that it is $\out(G; \mathcal{P})$ invariant. In particular, we are done by Proposition~\ref{lemma Mphi acting on T} and Corollary~\ref{cor action of Mphi} as long as $T$ is $G$-irreducible. 

However, if $T$ were reducible then either $G$ fixes a point of $T$ or $G$ admits an invariant line or end. If $G$ were to fix a point of $T$, then $G$ would equal a vertex stabiliser. Since $G$ is not elementary, the vertex stabiliser cannot be parabolic or virtually cyclic. And since $\out(G; \mathcal{P})$ is infinite, it cannot be rigid. Hence it must be a flexible quadratically hanging group with finite fibre. Therefore $G$ would be hyperbolic and we would be done by \cite[Theorem 3.5]{Levitt2000}. (This case could also be dealt with more concretely.) 

So we can assume that $G$ does not fix a point of $T$ and in particular $T$ is not a point and has an edge. We proceed to argue that $T$ does not admit an invariant end or line.  Consider $N$ the kernel of the action. If the action were reducible we would get that $G/N$ is either infinite cyclic or infinite dihedral by \cite[Corollary 2.3 and Theorem 2.5]{Culler1987}. (We note that in \cite{Culler1987}, those results construct homomorphisms to $\R$ and $Isom(\R)$, but in the context of a simplicial tree, these land in $\Z$ and $Isom(\Z)$ instead.). 

On the other hand, as in \cite[Corollary 9.20]{Guirardel2016}, edge stabilisers of $T$ are elementary meaning that they are either virtually cyclic or contained in (a conjugate of) a parabolic subgroup $P_i$. However, elementary subgroups are almost malnormal (for instance, see \cite[Corollary 3.2]{Guirardel2015}), which implies that $N$ is finite and if $T$ were reducible, then $G$ would have 0 or 2 ends and hence be elementary. Therefore $T$ is $G$-irreducible and  $G$ has infinitely many twisted conjugacy classes with respect to any $\Phi \in \out(G; \mathcal{P})$. 
\end{proof}


\relhyp

\begin{proof}
 Since the peripheral subgroups are not relatively hyperbolic, they are $\aut(G)$ invariant by \cite[Lemma 3.2]{Minasyan2012} and so $\out(G; \mathcal{P})=\out(G)$. 
 Therefore $G$ has property $R_{\infty}$ by Theorem~\ref{rel hyp and peripheral}. 
\end{proof}

\section{Appendix A:  \texorpdfstring{$R_{\infty}$ } \  for the infinite Dihedral Group  \texorpdfstring{$D_{\infty}$} \ }

\label{section D infinity}

The following result is well-known but we include it for completeness. See \cite[Proposition 2.3]{Goncalves2009}.

Note that the infinite Dihedral group is the only group which can be written as a non-trivial free product but having finitely many ends (it has two ends). 

\subsection{The infinite dihedral group}
The group $D_{\infty}$ is the free product of two cyclic groups of order 2 and hence has presentation, $$\langle x, y \ | \ x^2=y^2=1 \rangle.$$ 
	
	However, we will use the alternate generating set $\{ x, t=xy  \}$, which gives the corresponding presentation, $$\langle x, t \ | \ x^2=1, t^x= t^{-1}  \rangle. $$ This second presentation expresses the fact that $D_{\infty}$ is a semi-direct product $\Z \rtimes \Z_2$, where the monodromy is given by the (unique) non-identity automorphism of $\Z$. 
	
	With this second description is is straightforward to see that every element of $D_{\infty}$ can be written as either $t^n$ or $t^n x = x t^{-n}$ for some $n \in \Z$, and that this representation is unique. 

\begin{lemma}
\label{order 2}
    The outer automorphism group of the infinite dihedral group, $D_{\infty} = \langle x,y \ | \ x^2=y^2=1 \rangle$ has order $2$. The non-identity outer automorphism has a representative given by the map: 
    $$
    x \mapsto y, y \mapsto x. 
    $$
\end{lemma}
    
\begin{rem}
    It is easy to see from the universal property of free products that the map above defines an endomorphism of $D_{\infty}$ whose square is the identity. Hence it defines an automorphism. 
\end{rem}
\begin{proof}
    It is convenient to work with the alternate generating set, $x, t = xy$. First note that every nontrivial element of $\langle t\rangle$ has infinite order whereas every element of the coset $x\langle t \rangle$ has order $2$. Therefore any automorphism of $D_{\infty}$ must restrict to an automorphism of $\langle t \rangle$ and hence send $t$ to $t^{\pm 1}$.  

Notice that since $t^{-1} = t^x$ we may assume, up to inner automorphisms, that any automorphism fixes $t$. Also, since it has order 2, the image of $x$ must therefore be $x t^n$ for some $n \in \Z$. However, $(xt^n)^{t^k} = x t^{n+2k}$ and since $t^{t^k} = t$ we deduce that, up to inner automorphisms, any automorphism is represented by the maps $t \mapsto t, x \mapsto x$ and $t\mapsto t, x\mapsto xt$. It is a simple calculation to see that these maps are in the same outer automorphism classes as the identity map and the map given in the statement of this Lemma, respectively.
\end{proof}

\begin{thm} \label{D infty}
	The group $D_{\infty}$ has the $R_{\infty}$ property. 
\end{thm}
\begin{proof}

	The outer automorphism group of $D_{\infty}$ has order 2 by Lemma~\ref{order 2}, so by Lemma~\ref{lemma inner autos preserve twisted conjugacy}, it is sufficient to check that two automorphisms which represent these classes have infinitely many twisted conjugacy classes.

	\medspace
	
	\noindent
	\underline{Case 1: $\varphi$ is the identity} 
	
	When $\varphi$ is the identity, twisted conjugacy is simply standard conjugacy. We then note that, 
	$$ (t^n)^t = t^n \quad \text{ and } \quad (t^n)^x = t^{-n}.$$
	
	It follows that $t^n$ and $t^m$ are conjugate if and only if $|m|=|n|$. Hence, in this case, the infinitely many elements $t^n$, where $n$ is a positive integer, are all in different (twisted) conjugacy classes. 
	
	\medspace
	
	\noindent
	\underline{Case 2: $t \varphi = t^{-1}, x\varphi  = xt$. }
	
	This automorphism can be seen from the other presentation as the one which interchanges $x$ and $y$. In particular it is not inner and so represents the other outer automorphism class. Here we note that
	$$ (t \varphi)  (t^n x) t^{-1} =  t^{-1} t^n x t^{-1}  =    t^n x \quad \text{ and } \quad  (x \varphi) (t^n x ) x = x t^{n+1}   = t^{-n-1} x.$$
	
	It follows that $t^n x$ is twisted $\varphi$-conjugate to $t^m x$ if and only if $m=n$ or $m=-n-1$. Hence in this case, the infinitely many elements $t^n x$ where $n$ is a positive integer are all in different twisted conjugacy classes.

\end{proof}

\section{Appendix B: Property \texorpdfstring{$R_\infty$ }\ via quasimorphisms}

\subsection{Quasimorphisms}

In this appendix we present a more indirect approach to prove property $R_\infty$, which applies to the groups considered in this paper, and more.

\begin{defn}
    Let $G$ be a group, and $f \colon G \to \mathbb{R}$ a function. The \emph{defect} of $f$ is
    \[D(f) \coloneqq \sup\limits_{x, y \in G} |f(xy) - f(x) - f(y)|.\]
    If $D(f) < \infty$, we say that $f$ is a \emph{quasimorphism}. If moreover $f(x^n) = n f(x)$ for all $x \in G$ and all $n \in \mathbb{Z}$, we say that $f$ is \emph{homogeneous}.
\end{defn}

Quasimorphisms are useful tools for the interactions of group theory with various subjects, such as bounded cohomology \cite{fff-frigerio}, stable commutator length \cite{fff-calegari}, knot theory \cite{fff-knot}, symplectic geometry \cite{fff-symplectic} and dynamics \cite{fff-ghys}. Quasimorphisms are abundant among groups with hyperbolic features: this originates from the Brooks construction for free groups \cite{fff-brooks}, and culminated in the Bestvina--Fujiwara construction for acylindrically hyperbolic groups \cite{fff-BF}. Here we are only concerned with quasimorphisms with an additional special property.

\begin{defn}
    Let $f \colon G \to \mathbb{R}$ be a quasimorphism and $\varphi \in \aut(G)$. We say that $f$ is \emph{$\varphi$-invariant} if $f(x\varphi) = f(x)$ for all $x \in G$. We say that $f$ is \emph{$\aut$-invariant} if it is $\varphi$-invariant for every $\varphi \in \aut(G)$.
\end{defn}

Every quasimorphism is at a bounded distance from a homogeneous one \cite[Lemma 2.21]{fff-calegari}, so it is common to restrict to those. Homogeneous quasimorphisms are easily seen to be conjugacy invariant \cite[Section 2.2.3]{fff-calegari}, therefore $\varphi$-invariance of a homogeneous quasimorphism is really a property of the corresponding outer automorphism.

\begin{thm}
\label{fff-local criterion}
    Let $G$ be a group, let $\varphi \in \aut(G)$, and suppose that there exists a non-zero $\varphi$-invariant homogeneous quasimorphism on $G$. Then $G$ has property $R_\infty(\varphi)$.
\end{thm}

\begin{proof}
    Let $f \colon G \to \mathbb{R}$ be as in the statement. Let $\thickapprox_D$ denote an equality up to $D(f)$, which is finite by assumption. Then we estimate:
    \[f(wx(w\varphi)^{-1}) \thickapprox_D f(wx) + f((w\varphi)^{-1}) \thickapprox_D f(w) + f(x) + f((w\varphi)^{-1}) = f(x),\]
    where we used homogeneity and $\varphi$-invariance in the last equality. In other words:
    \[x \sim_\varphi y \qquad \Longrightarrow \qquad|f(x) - f(y)| \leq 2 D(f).\]
    Now let $x \in G$ be such that $f(x) \neq 0$: this exists by assumption. Up to replacing $x$ by a suitable power, using homogeneity, we may assume that $|f(x)| > 2 D(f) \geq 0$. Then for all $i \neq j \in \mathbb{Z}:$
    \[|f(x^{i}) - f(x^{j})| = |i-j| |f(x)| \geq |f(x)| > 2 D(f),\]
    and so $x^i$ and $x^j$ cannot be in the same $\sim_\varphi$-class.
\end{proof}

We immediately obtain:

\begin{cor}
\label{fff-global criterion}
    If there exists a non-zero $\aut$-invariant homogeneous quasimorphism on $G$, then $G$ has property $R_\infty$. \qed
\end{cor}

Let us note that the proof is showing more: if there exists a non-zero $\varphi$-invariant homogeneous quasimorphism, then there exists $x \in G$ whose powers are in pairwise distinct $\sim_\varphi$ classes. Similarly, if there exists a non-zero $\aut$-invariant homogeneous quasimorphism, then there exists $x \in G$ whose powers are in pairwise distinct $\sim_\varphi$ classes, for all $\varphi$ simultaneously.

\begin{rem}
    Of course homomorphisms are homogeneous quasimorphisms, therefore our result recovers \cite[Theorem 5.4]{fff-thompsonlike}, which states that groups with non-zero $\aut$-invariant homomorphisms have property $R_\infty$. There, it was applied to prove property $R_\infty$ for a large family of Thompson-like groups, including the piecewise linear $F$-like groups of Bieri--Strebel \cite{fff-bieristrebel}, the piecewise projective groups of Lodha--Moore \cite{fff-lodhamoore}, and the braided Thompson group $F_{br}$ \cite{fff-braidedF}. In the first two cases, our more general criterion does not help: if $G$ is a piecewise linear or piecewise projective group of the line, then every homogeneous quasimorphism is a homomorphism \cite{fff-PL, fff-spectrum}. In the third case, there exist many homogeneous quasimorphisms that are not homomorphisms \cite{fff-braidedqm}, but the construction does not give information about $\aut$-invariance. The interest of our more general criterion is for applications to groups with hyperbolic features: already in free groups, it is easy to see that there exist no $\aut$-invariant homomorphisms, but $\aut$-invariant homogeneous quasimorphisms do exist \cite{fff-wade}.
\end{rem}

Such quasimorphisms have been constructed in several cases. The most widely applicable criterion is the following:

\begin{thm}[{\cite[Theorem E]{fff-wade}}]
\label{fff-wade}
    Let $G$ be a group such that $\inn(G)$ is infinite and $\aut(G)$ is acylindrically hyperbolic. Then there exists an infinite-dimensional space of $\aut$-invariant homogeneous quasimorphisms.
\end{thm}

Acylindrical hyperbolicity of automorphism groups has been proved in several cases (in fact, it is still an open question whether it holds for \emph{all} finitely generated acylindrically hyperbolic groups \cite[Question 1.1]{fff-oneend}), and so we obtain the following list of examples of groups with property $R_\infty$:


\generalrinfty

\begin{proof}
    Groups in the second item have acylindrically hyperbolic automorphism group \cite[Theorem 1.3]{fff-infend}. Therefore, they admit non-zero $\aut$-invariant quasimorphisms by Theorem \ref{fff-wade}, and hence have property $R_\infty$ by Corollary \ref{fff-global criterion}. The groups in the first and third item are special cases (see also \cite{fff-oneend} and \cite[Theorem 1.1]{fff-infend}). Similarly, the groups in the fourth item have acylindrically hyperbolic automorphism group \cite[Theorem A.27]{fff-graph:EH} (see also \cite{fff-graph:GM} and \cite{fff-graph:genevois}), with the exception of the infinite dihedral group, which is treated separatedly in Theorem \ref{D infty}. The groups in the fifth item need not have acylindrically hyperbolic automorphism group (because we are allowing joins), but they still admit $\aut$-invariant homogeneous quasimorphisms \cite[Theorem B(4)]{fff-wade}, so they have property $R_\infty$ by Corollary \ref{fff-global criterion}.
\end{proof}

The first item recovers \cite{Levitt2000}. The second item recovers Theorem \ref{rel hyp general} (and removes the hypothesis of finite presentability). The third item recovers Theorem \ref{inf ends has rinf} (and removes the hypothesis of accessibility). The last item recovers the $R_\infty$ property for non-abelian right angled Artin groups \cite{Witdouck2023}. Corollary \ref{fff-global criterion} has since been applied to prove property $R_\infty$ for a large class of Artin groups, see \cite[Corollary F]{jsj:artin}.

\begin{rem}
    The same argument as Theorem \ref{fff-wade} (see \cite[Corollary 4.4]{fff-wade} and the proof of \cite[Theorem 5.1]{fff-wade}) shows that, if $\inn(G)$ is infinite, and $\langle \inn(G), \varphi \rangle < \aut(G)$ is acylindrically hyperbolic, then $G$ admits an infinite-dimensional space of $\varphi$-invariant homogeneous quasimorphism. By Theorem \ref{fff-local criterion}, this statement is enough to deduce property $R_\infty(\varphi)$. It is possible that, for some groups, proving that $\langle \inn(G), \varphi \rangle$ is acylindrically hyperbolic for all $\varphi$ is easier than proving that $\aut(G)$ itself is acylindrically hyperbolic.
\end{rem}


\end{document}